\documentclass[12pt,leqno,fleqn]{amsart}
\usepackage{amsmath,amstext,amsthm,amssymb,amsxtra}
\usepackage[top=1.5in, bottom=1.5in, left=1.25in, right=1.25in]	{geometry}
\usepackage[normalem]{ulem}
\usepackage{txfonts} 
\usepackage[T1]{fontenc}
\usepackage{lmodern}
\usepackage{tikz}\usetikzlibrary{arrows}

\usepackage{euler}   
\usepackage{enumerate}

\makeatletter
\DeclareRobustCommand\widecheck[1]{{\mathpalette\@widecheck{#1}}}
\def\@widecheck#1#2{%
    \setbox\z@\hbox{\m@th$#1#2$}%
    \setbox\tw@\hbox{\m@th$#1%
       \widehat{%
          \vrule\@width\z@\@height\ht\z@
          \vrule\@height\z@\@width\wd\z@}$}%
    \dp\tw@-\ht\z@
    \@tempdima\ht\z@ \advance\@tempdima2\ht\tw@ \divide\@tempdima\thr@@
    \setbox\tw@\hbox{%
       \raise\@tempdima\hbox{\scalebox{1}[-1]{\lower\@tempdima\box
\tw@}}}%
    {\ooalign{\box\tw@ \cr \box\z@}}}
\makeatother

\usepackage{mathtools}
\mathtoolsset{showonlyrefs,showmanualtags}

\usepackage{hyperref} 
\hypersetup{
    colorlinks=true,       
    linkcolor=blue,          
    citecolor=magenta,        
    filecolor=magenta,      
    urlcolor=cyan           
}

\theoremstyle{plain} 
\newtheorem{lemma}[equation]{Lemma}

\newtheorem{theorem}[equation]{Theorem}
\newtheorem{corollary}[equation]{Corollary}

\theoremstyle{definition}
\newtheorem{definition}[equation]{Definition}

\theoremstyle{remark}
\newtheorem{remark}[equation]{Remark}

\numberwithin{equation}{section}

\title[bilinear maximal operator on filtered measure spaces]
  {Weighted estimates for the bilinear maximal operator on filtered measure spaces}


\author[W. Chen]{Wei Chen}
\address{School of Mathematical Sciences, Yangzhou University, Yangzhou 225002, China}
\email {weichen@yzu.edu.cn}

\author[Y. Jiao]{Yong Jiao}
\address{School of Mathematics and Statistics, Central South University, 410075 Changsha, China}
\email{jiaoyong@csu.edu.cn}

\thanks{The research of Wei Chen is supported by NSFC (11971419, 11771379) and the School Foundation of Yangzhou University (2019CXJ001). The research of Yong Jiao is supported by NSFC (11722114,11961131003).}

\begin{document}
\begin{abstract} Assuming the bilinear reverse H\"{o}lder's condition, we character weighted inequalities for the bilinear maximal
operator on filtered measure spaces. We also obtain Hyt\"{o}nen-P\'{e}rez type weighted estimates for the bilinear maximal
operator. Our approaches are mainly based on the new construction of bilinear versions of principal sets and the new Carleson embedding theorem on filtered measure spaces.
In particular, we find a new property of the construction and we call it the conditional sparsity of principal sets.
\end{abstract}

	\maketitle
\tableofcontents

\section{Introduction}\label{Tntro}
Let $\mathbb R^n$ be the $n\hbox{-dimensional}$ real Euclidean
space and $f$ a real valued measurable function, the classical Hardy-littlewood maximal operator is defined by
$$Mf(x)=\sup\limits_{x\in Q}\frac{1}{|Q|}\int_Q|f(y)|dy,$$
 where $Q$ is a non-degenerate cube with its sides parallel to the coordinate
axes and $|Q|$ is the Lebesgue measure of $Q.$

Let $u,~v$ be two weights, i.e. positive measurable functions. As is
well known, for $p\geq1,$ Muckenhoupt \cite{Muckenhoupt B} showed that the
inequality
$$ \lambda^p\int_{\{Mf>\lambda\}}u(x)dx
\leq C \int_{\mathbb R^n}|f(x)|^pv(x)dx, ~~\lambda>0,~f\in{L^p(v)}
$$
holds if and only if $(u,v)\in A_p,$ i.e., for any cube $Q$ in
$\mathbb R^n$ with sides parallel to the coordinates
$$
\Big(\frac{1}{|Q|}\int_Qu(x)dx\Big)
\Big(\frac{1}{|Q|}\int_Qv(x)^{-\frac{1}{p-1}}dx\Big)^{p-1}<C,~p>1;
$$
$$
\frac{1}{|Q|}\int_Qu(x)dx\leq C \mathop{\hbox{ess inf}}\limits_{Q}v(x)
,~p=1.
$$
Suppose that $u=v$ and $p>1,$ Muckenhoupt \cite{Muckenhoupt B} also proved that
$$ \int_{\mathbb R^n}\big(Mf(x)\big)^pv(x)dx
\leq C \int_{\mathbb R^n}|f(x)|^pv(x)dx, ~\forall f\in{L^p(v)}
$$
holds if and only if $v$ satisfies
\begin{equation}\label{Ap}
\Big(\frac{1}{|Q|}\int_Qv(x)dx\Big)
\Big(\frac{1}{|Q|}\int_Qv(x)^{-\frac{1}{p-1}}dx\Big)^{p-1}<C,~\forall Q\subset \mathbb R^n.
\end{equation}
The crucial step is to show that if $v$ satisfies $A_p$,
then there is an $\varepsilon> 0$ such that $v$
also satisfies $A_{p-\varepsilon}.$
However, the problem of finding all $u$ and $v$ such that
$$ \int_{\mathbb R^n}\big(Mf(x)\big)^pu(x)dx
\leq C \int_{\mathbb R^n}|f(x)|^pv(x)dx,~\forall f\in{L^p(v)}
$$
is much hard and complicated. In order to solve the problem,
Sawyer \cite{Sawyer E T.} established the
testing condition $S_{p,q},$ i.e. for any cube $Q$ in
$\mathbb{R}^n$ with sides parallel to the coordinates
$$\Big(\int_{Q}\big(M(\chi_Qv^{1-p'})(x)\big)^qu(x)dx\Big)^{\frac{1}{q}}
\leq C\Big(\int_Qv(x)^{1-p'}dx\Big)^{\frac{1}{p}},~\forall Q\subset \mathbb R^n,$$
where $1<p\leq q<\infty$ and $\chi_Q$ is the characteristic function of $Q.$
The condition $S_{p,q}$ is a
sufficient and necessary condition such that
the weighted inequality
$$\Big(\int_{\mathbb R^n}\big(Mf(x)\big)^qu(x)dx\Big)^{\frac{1}{q}} \leq C
\Big(\int_{\mathbb R^n}|f(x)|^pv(x)dx\Big)^{\frac{1}{p}}, ~\forall f\in{L^p(v)}$$
holds.
In this case, the method of proof is very interesting.
Motivated by \cite{Muckenhoupt B, Sawyer E T.},
the theory of weights developed so
rapidly that it is difficult to give its history a full
account here (see \cite{Garcia Rubio}
and \cite{Cruz-Uribe D. J. M. Martell} for more information).

Weighted estimates for the maximal operator $\prod_{j=1}^m Mf_j$ ($m$-fold product of $M$) in the multilinear setting were studied in \cite{GT} and \cite{PT}. Recently, the new multilinear maximal function
\begin{equation}\label{multi_maximal_operator}\mathcal{M}(f_1,...,f_m)(x) := \sup\limits_{x\in Q}
\prod\limits_{i=1}\limits^{m}\frac{1}{|Q|}\int_Q|f_i(y_i)|dy_i, \quad x\in \mathbb R^n
\end{equation}
associated with cubes with sides parallel to the coordinate
axes was first defined and the corresponding weight theory was studied in \cite{Lerner A.K. Ombrosi S.}.
The importance of this operator is that it is strictly smaller than the $m$-fold product of $M$. Moreover, it generalizes the Hardy--Littlewood
maximal function (case $m=1$) and in several ways it controls the class
of multilinear Calder\'{o}n--Zygmund operators as shown in \cite{Lerner A.K. Ombrosi S.}.
The relevant class of multiple weights for $\mathcal{M}$ is given by the condition $A_{\overrightarrow{p}}$ \cite[Definition 3.5]{Lerner A.K. Ombrosi S.}: for
$\overrightarrow{p}=(p_1,p_2,\cdot\cdot\cdot,~ p_m),$
$\overrightarrow{\omega}=(\omega_1, ~\omega_2,\cdot\cdot\cdot,~\omega_m)$ and a weight $v,$
the weight vector $(v, \overrightarrow{\omega})\in A_{\overrightarrow{p}}$ if
$$\sup_Q\Big(\frac{1}{|Q|}\int_Qv(x)dx\Big)\prod\limits^m_{i=1}\Big(\frac{1}{|Q|}\int_Q\omega_i(x)^{-\frac{1}{p_i-1}}dx\Big)^{\frac{p}{p'_i}}
< \infty,$$
where $\frac{1}{p}=\sum\limits^m_{i=1}\frac{1}{p_i }$ and $1\leq p_1,p_2,...,p_m<\infty.$
The more general case was extensively discussed in
\cite{Grafakos L. Liu L. G. Yang D. C., Grafakos L. Liu L. G. Perez C. Torres R. H.}.
Using a dyadic discretization technique, Dami\'{a}n, Lerner and P\'{e}rez \cite{W. Damian}
and Li, Moen and Sun \cite{Li Moen Sun} proved some sharp weighted norm inequalities for the multilinear maximal operator $\mathcal M.$
In order to establish the generalization of Sawyer's theorem to the multilinear setting,
a kind of monotone property and a
reverse H\"{o}lder's condition on the weights were introduced in \cite{W. M. Li} and \cite{Chen-Damian}, respectively. Note that if $v=\prod\limits_{i=1}^m\omega_i^{\frac{p}{p_i}},$ then the condition
$(v,\overrightarrow{\omega})\in A_{\overrightarrow{p}}$ implies the reverse H\"{o}lder's condition $\overrightarrow{\omega}\in RH_{\overrightarrow{p}}$ \cite[Proposition 2.3]{cao-xue}.
In addition, Chen and Dami\'{a}n investigated a bound $B_{\overrightarrow{p}}$ \cite[Theorem 2]{Chen-Damian} and a mixed
bound $A_{\overrightarrow{p}}-W^{\infty}_{\overrightarrow{p}}$ \cite[Theorem 3]{Chen-Damian} for the multilinear maximal operator, which are the multilinear versions of one-weight norm estimates \cite[Theorem 4.3]{HP}.
Still more recently, the multilinear fractional maximal operator and the multilinear fractional strong maximal operator associated with rectangles were studied in \cite{cao-xue} and \cite{cao-xue-yabuta}, respectively.

On the other hand, Tanaka and Terasawa \cite{TT} very recently developed a theory of weights for positive (linear) operators and the generalized Doob's maximal operators on a filtered measure space. In particular, two-weight norm inequalities \cite[Theorem 4.1]{TT} and one-weight norm estimates of Hyt\"{o}nen-P\'{e}rez type \cite[Theorem 5.1]{TT} for Doob's maximal operator were established by the use of the Carleson embedding theorem and the construction of principal set, respectively. Note that a filtered measure space naturally contains a filtered probability space with a filtration indexed by $\mathbb N$ and a Euclidean space with a dyadic filtration. It also contains a doubling metric measure space with dyadic lattice constructed by Hyt\"{o}nen and Kairema \cite{HK}. From this perspective, Dyadic Harmonic Analysis on the Euclidean space and Martingale Harmonic Analysis on a probability space can be unified on a filtered (infinite) measure space, as treated in \cite{H_Kem, Schilling, Stroock}.
We also mention that the Haar shift  operators were studied by Lacey,
Petermichl and Reguera in \cite{Lacey} and played an important role in the resolution of the so-called
$A_2$ conjecture in \cite{Hyt}. On a filtered measure space, these operators could be seen from the point-of-view of martingale theory.

Motivated by the works above, the purpose of this paper is to develop a theory of weights for multilinear Doob's maximal operator on a filtered measure space. For simplicity of notations, we only consider the bilinear case and all results can be extended to the multilinear case without essential difficulty.

The following theorem is our first main result, which
gives the weights for which the bilinear maximal
operator $\mathcal{M}$ is bounded from
$L^{p_1}(\omega_1)\times L^{p_2}(\omega_2 )$
to $L^p(\omega_1^{\frac{p}{p_1}}\omega_2^{\frac{p}{p_2}}).$ All unexplained notations can be found in Section \ref{Pre}.

\begin{theorem} \label{thm_AP}Let $\omega_1, ~\omega_2$ be weights and $1< p_1, ~p_2<\infty.$
Suppose that
$$v=\omega_1^{\frac{p}{p_1}}\omega_2^{\frac{p}{p_2}}\quad {\rm and} \quad \sigma_i=\omega_i^{-\frac{1}{p_i-1}}, ~i=1,~2, \quad 1/p=1/p_1+1/p_2.$$
\begin{enumerate}[\rm(1)]
\item  \label{re1} If $(v,~\omega_1, ~\omega_2)\in A_{\overrightarrow{p}}$, then there exists a positive constant $C$ such that
for all $f_1\in L^{p_1}(\sigma_1),~f_2\in L^{p_2}(\sigma_2)$ we have
      \begin{equation}
\label{Th_AA_1}\|\mathcal{M}(f_1\sigma_1,f_2\sigma_2)\|_{L^p(v)}\leq
C\|f_1\|_{L^{p_1}(\sigma_1)}\|f_2\|_{L^{p_2}(\sigma_2)}.
\end{equation}
We denote the smallest constant $C$ in \eqref{Th_AA_1} by $\|\mathcal{M}\|.$
Then it follows that
$\|\mathcal{M}\|\leq16\cdot4^{(q'-1)}p'_1p'_2[v,\omega_1,\omega_2]^{\frac{q'}{p}}_{A_{\overrightarrow P}},$ where $q=\min\{p_1,p_2\}.$
\item  \label{re2} Let $(\omega_1, \omega_2)\in RH_{\overrightarrow{p}}.$ If there exists a positive constant $C$ such that for all $
f_1\in L^{p_1}(\sigma_1),$ $f_2\in L^{p_2}(\sigma_2)$ we have
      \begin{equation}
\label{Th_AB_1}\|\mathcal{M}(f_1\sigma_1,f_2\sigma_2)\|_{L^p(v)}\leq
C\|f_1\|_{L^{p_1}(\sigma_1)}\|f_2\|_{L^{p_2}(\sigma_2)},
\end{equation}
then $(v,~\omega_1, ~\omega_2)\in A_{\overrightarrow{p}}.$ We denote the smallest constant $C$ in \eqref{Th_AB_1} by $\|\mathcal{M}\|.$
Then it follows that
$[v,\omega_1,\omega_2]_{A_{\overrightarrow P}}\leq\|\mathcal{M}\|^p[\omega_1, \omega_2]_{RH_{\overrightarrow{p}}}.$
\end{enumerate}
\end{theorem}

First, Theorem \ref{thm_AP} (more precisely, Corollary \ref{cor_AP} below) is a bilinear analogue of \cite[Corollary 4.5]{TT}. We remark that the reverse H\"{o}lder's condition $RH_{\vec{p}}$ is automatically true in the linear case. Second, our theorem is an extension of \cite[Theorem 3.7]{Lerner A.K. Ombrosi S.} and \cite[Theorem 1.2]{Li Moen Sun} to a filtered measure space.
In \cite{Lerner A.K. Ombrosi S.} they showed that the multilinear $A_{\vec{p}}$ condition has interesting characterization in terms of the linear $A_p$ classes \cite[Theorem 3.6]{Lerner A.K. Ombrosi S.}. Then their proof was based on the Reverse H\"{o}lder's inequality for linear $A_p$ classes. However, they are invalid on a filtered measure space (even on a filtered probability spaces without regularity condition \cite[p.262]{R. L. Long}). Li, Moen and Sun \cite[Theorem 1.2]{Li Moen Sun} found the optimal power on $[v,\omega_1,\omega_2]_{A_{\overrightarrow P}}$ and their proof depends very much on the properties of the sparse family on Euclidean spaces.
Because a filtered measure space contains a Euclidean space with a dyadic filtration, our theorem gives sharp bound for the bilinear maximal operator $\mathcal{M}.$
Motivated by \cite[Theorem 1.2]{Li Moen Sun}, our proof is mainly based on the bilinear construction of principal sets on filtered measure spaces.
The germ of principal set appeared as the sparse family on $\mathbb R^n$ (see \cite{HP,W. Damian} for more information) and was successfully constructed on the filtered measure space in \cite[pp.942-943]{TT}.
We find a new property (Section \ref{constru}, P.\ref{new}) of the construction and we call it the conditional sparsity of principal sets.

Theorem \ref{thm_AP} also completes the bilinear version of one-weight theory in the martingale setting \cite[Proposition 1.15]{CL1}. In fact, in \cite{CL1} only the second part of Theorem \ref{thm_AP} on a filtered probability space was proved. In addition, it is clear that \eqref{Th_AA_1} implies the condition $S_{\overrightarrow{p}}.$
Then it follows from Theorem \ref{thm_AP} that the condition $A_{\overrightarrow{p}}$ implies the condition $S_{\overrightarrow{p}}.$ Hence, Theorem \ref{thm_AP} is a bilinear analogue of the equivalence between $A_p$ and $S_p$ in \cite{Hunt}.

Our second main purpose is to character two-weight inequalities for the bilinear maximal
operator on the filtered measure space. Assuming the reverse H\"{o}lder's condition,  Theorem \ref{thm_Sp} below is a bilinear version of Sawyer's result \cite[Theorem A]{Sawyer E T.} on filtered measure spaces.

\begin{theorem}\label{thm_Sp} Let $v,\omega_1, \omega_2$ be weights and $1<p_1, p_2<\infty.$
Suppose that
$1/p=1/p_1+1/p_2$ and $(\omega_1, \omega_2)\in RH_{\overrightarrow{p}},$
then the following statements
are equivalent
\begin{enumerate}[$\rm(1)$]
\item \label{Spa} There exists a positive constant $C$ such that for all $f_1\in L^{p_1}(\sigma_1),~f_2\in L^{p_2}(\sigma_2)$ we have
\begin{equation}\label{thmSp}
\|\mathcal{M}(f_1\sigma_1,f_2\sigma_2)\|_{L^p(v)}\leq
C\|f_1\|_{L^{p_1}(\sigma_1)}\|f_2\|_{L^{p_2}(\sigma_2)}
\end{equation}
where $\sigma_i=\omega_i^{-\frac{1}{p_i-1}}, ~i=1,~2.$
\item \label{Spb} The triple of weights $(v,~\omega_1, ~\omega_2)$
satisfies the condition $S_{\overrightarrow{p}}.$
               \end{enumerate}
Moreover, we denote the smallest constant $C$ in (\ref{thmSp})
by $\|\mathcal{M}\|.$
Then it follows that
$$[v,\overrightarrow{\omega}]_{S_{\overrightarrow{p}}}
\leq\|\mathcal{M}\|\leq 32p'_1p'_2
[v,\overrightarrow{\omega}]_{S_{\overrightarrow{p}}}[\omega_1, \omega_2]^{\frac{1}{p}}_{RH_{\overrightarrow{p}}}.$$
\end{theorem}

We also obtain Hyt\"{o}nen-P\'{e}rez type weighted estimates \cite[theorem 4.3]{HP} for the bilinear maximal operator on filtered measure spaces. To be precise, we prove the following Theorems \ref{theorem_Bp} and \ref{theorem_bi A Fujii}. Their linear cases were studied in \cite[Theorem 1.7]{CZZJ2020}.

\begin{theorem}\label{theorem_Bp}
If $(v,\overrightarrow{\omega})\in B_{\overrightarrow{p}},$ then the following statements
are valid:\begin{enumerate}[\rm(1)]
\item \label{thm B 1}There exists a positive constant $C$ such that
for all $f\in L^{p_1}(\omega_1),~g\in L^{p_2}(\omega_2)$ we have
\begin{equation}\label{BPaa}
\|\mathcal{M}(f_1,f_2)\|_{L^p(v)}\leq
C\|f_1\|_{L^{p_1}(\omega_1)}\|f_2\|_{L^{p_2}(\omega_2)}.
\end{equation}
\item \label{thm B 2}There exists a positive constant $C$ such that for all $f\in L^{p_1}(\sigma_1),~g\in L^{p_2}(\sigma_2)$ we have
\begin{equation}\label{BPbb}
\|\mathcal{M}(f_1\sigma_1,f_2\sigma_2)\|_{L^p(v)}\leq
C\|f_1\|_{L^{p_1}(\sigma_1)}\|f_2\|_{L^{p_2}(\sigma_2)}.
\end{equation}
\end{enumerate}
Moreover, we denote the smallest constants $C$ in \eqref{BPaa}
and \eqref{BPbb} by $\|\mathcal{M}\|$ and $\|\mathcal{M}\|',$ respectively.
Then it follows that $\|\mathcal{M}\|=\|\mathcal{M}\|'\leq32(2e)^{\frac{1}{p}}p'_1p'_2
[v,\overrightarrow{\omega}]_{B_{\overrightarrow{p}}}^{\frac{1}{p}}.$
\end{theorem}

\begin{theorem}\label{theorem_bi A Fujii}
If $(v,\omega_1, \omega_2)\in A_{\overrightarrow{p}}$
and $(\omega_1, \omega_2)\in W^{\infty}_{\overrightarrow{p}},$ then the following statements
are valid:\begin{enumerate}[\rm(1)]
\item \label{thm BC 1}There exists a positive constant $C$ such that for all $f\in L^{p_1}(\omega_1),~g\in L^{p_2}(\omega_2)$ we have
\begin{equation}\label{BPaaC}
\|\mathcal{M}(f_1,f_2)\|_{L^p(v)}\leq
C\|f_1\|_{L^{p_1}(\omega_1)}\|f_2\|_{L^{p_2}(\omega_2)}.
\end{equation}
\item \label{thm BC 2}There exists a positive constant $C$ such that for all $f\in L^{p_1}(\sigma_1),~g\in L^{p_2}(\sigma_2),$ we have
\begin{equation}\label{BPbbC}
\|\mathcal{M}(f_1\sigma_1,f_2\sigma_2)\|_{L^p(v)}\leq
C\|f_1\|_{L^{p_1}(\sigma_1)}\|f_2\|_{L^{p_2}(\sigma_2)}.
\end{equation}
\end{enumerate}
Moreover, we denote the smallest constants $C$ in \eqref{BPaaC}
and \eqref{BPbbC} by $\|\mathcal{M}\|$ and $\|\mathcal{M}\|',$ respectively.
Then it follows that $$\|\mathcal{M}\|=\|\mathcal{M}\|'
\leq32\cdot2^{\frac{1}{p}}p'_1p'_2 [v,\omega_1, \omega_2]_{A_{\overrightarrow{p}}}^{\frac{1}{p}}
[\omega_1, \omega_2]_{W_{\overrightarrow{p}}^\infty}^{\frac{1}{p}}.$$
\end{theorem}

\begin{remark}\label{chen-wendy} Using a dyadic discretization technique, Chen and Dami\'{a}n \cite{Chen-Damian} investigated Theorems \ref{thm_Sp}, \ref{theorem_Bp} and \ref {theorem_bi A Fujii} on $\mathbb{R}^n.$ In addition, Cao and Xue \cite{cao-xue} and Sehba \cite{Sehba} gave the similar theorems for bilinear fractional maximal function on $\mathbb{R}^n,$ respectively.
\end{remark}

To prove Theorems \ref{thm_Sp}, \ref{theorem_Bp} and \ref {theorem_bi A Fujii},
the key ingredient is the bilinear version of Carleson embedding theorem
associated with the collection of principal sets developed in Section \ref{Car}.
Note that Hyt\"{o}nen and P\'{e}rez gave the dyadic Carleson embedding theorem \cite[Theorem 4.5]{HP} (see \cite{Pereyra} for more information), and
Chen and Dami\'{a}n obtained its multilinear analogue \cite[Lemma 3]{Chen-Damian} on $\mathbb{R}^n.$ In order to provide some two-weight norm estimates for multilinear fractional maximal function, Sehba \cite{Sehba} extensively discussed the more general Carleson embedding theorem.
Tanaka and Terasawa \cite[Section 3]{TT} introduced a refinement of the Carleson embedding theorem on a filtered measure space.
In the present paper, our Carleson embedding theorem associated with the collection of principal sets is very different from \cite[Theorem 3.1]{TT}; see Theorem \ref{Carleson_thm} in Section \ref{Car}.

\begin{remark}\label{RH} As treated on filtered probability spaces \cite{CL1} and on Euclidean spaces \cite{cao-xue,Chen-Damian,Sehba}, we do not know if the reverse H\"{o}lder condition $RH_{\vec{p}}$ in the theorems above is essential.
\end{remark}

The article is organized as follows.
In Section \ref{Pre}, we state some preliminaries. We construct bilinear versions of principal sets
and Carleson embedding theorem in Section \ref{constru} and Section \ref{Car}, respectively. In Section \ref{proofs}, we provide the proofs of the above theorems.

The letter $C$ will be used for constants that may change from one occurrence to another.

\section{Preliminaries}\label{Pre}
This section consists of the preliminaries for this paper.
\subsection{bilinear maximal operator on filtered measure spaces}
In this subsection we introduce the bilinear maximal operator on filtered measure spaces, which are standard \cite{TT}. Let a triplet $(\Omega,\mathcal{F},\mu)$ be a measure space. Denote by $\mathcal{F}^0$ the collection of sets in $\mathcal{F}$ with
finite measure. The measure space $(\Omega,\mathcal{F},\mu)$ is called $\sigma$-finite if there exist sets $E_i\in\mathcal{F}^0$ such
that $\Omega=\bigcup\limits_{i=0}^{\infty}E_i$. In this paper all measure spaces are assumed to be $\sigma$-finite. Let $\mathcal{A}\subset\mathcal{F}^0$ be
an arbitrary subset of $\mathcal{F}^0$. An $\mathcal{F}$-measurable function $f:\Omega\rightarrow \mathbb R$ is called $\mathcal{A}$-integrable if it is
integrable on all sets of $\mathcal{A}$, i.e.,
$\chi_Ef\in L^1(\mathcal{F},\mu)$ for all $E\in \mathcal{A}$.
Denote the collection of all such functions by $L^1_{\mathcal{A}}(\mathcal{F},\mu).$

If $\mathcal{G}\subset\mathcal{F}$ is another $\sigma$-algebra, it is called a sub-$\sigma$-algebra of $\mathcal{F}$. A function $g\in L^1_{\mathcal{G}^0}(\mathcal{G},\mu)$ is
called the conditional expectation of $f\in L^1_{\mathcal{G}^0}(\mathcal{F},\mu)$ with respect to $\mathcal{G}$ if there holds
\begin{equation*}
\int_Gfd\mu=\int_Ggd\mu,\quad \forall G\in\mathcal{G}^0.
\end{equation*}
The conditional expectation of $f$ with respect to $\mathcal{G}$ will be denoted by $\mathbb E(f|\mathcal{G})$, which exists
uniquely in $L^1_{\mathcal{G}^0}(\mathcal{G},\mu)$ due to $\sigma$-finiteness of $(\Omega,\mathcal{G},\mu).$

A family of sub-$\sigma$-algebras $(\mathcal{F}_i)_{i\in \mathbb Z}$ is called a filtration of $\mathcal{F}$ if $\mathcal{F}_i\subset\mathcal{F}_j\subset\mathcal{F}$ whenever
$i, j \in \mathbb Z$ and $i < j.$ We call a quadruplet $(\Omega,\mathcal{F},\mu; (\mathcal{F}_i)_{i\in \mathbb Z})$
a $\sigma$-finite filtered measure space. As remarked in Section 1, it
contains a filtered probability space with a filtration indexed by $\mathbb N,$
a Euclidean space with a dyadic filtration and doubling metric space with dyadic lattice.

We
write
\begin{equation*}
\mathcal{L}:=\bigcap\limits_{i\in \mathbb Z}L^1_{\mathcal{F}^0_i}(\mathcal{F},\mu).\end{equation*}
Notice that
\begin{equation*}
L^1_{\mathcal{F}^0_i}(\mathcal{F},\mu)\supset L^1_{\mathcal{F}^0_j}(\mathcal{F},\mu)\end{equation*} whenever $i < j.$
For a function $f\in\mathcal{L}$ we will denote
$\mathbb E(f|\mathcal{F}_i)$ by $\mathbb E_i(f).$ By the tower rule of conditional expectations, a family of functions
$\mathbb E_i(f)\in L^1_{\mathcal{F}^0_i}(\mathcal{F},\mu)$ becomes a martingale.

By a weight we mean a nonnegative function which belongs to $\mathcal{L}$ and, by a convention, we
will denote the set of all weights by $\mathcal{L}^+.$

Let $(\Omega,\mathcal{F},\mu; (\mathcal{F}_i)_{i\in \mathbb Z})$  be a $\sigma$-finite filtered measure space. Then a function
$\tau:~ \Omega\rightarrow\{-\infty\}\cup \mathbb Z\cup\{+\infty\}$ is called a stopping time if for any $i\in \mathbb Z,$ we have
$\{\tau=i\}\in \mathcal{F}_i.$ The family of all stopping times is denoted by $\mathcal{T}.$
Fixing $i\in \mathbb Z,$ we denote $\mathcal{T}_i=\{\tau\in \mathcal{T}:~\tau\geq i\}.$

Suppose that functions $f\in\mathcal{L}$ and $g\in\mathcal{L},$ the maximal operator
and bilinear maximal operator are defined by
$$Mf=\sup_{i\in \mathbb Z}|\mathbb E_i(f)|\text{ and }
\mathcal{M}(f,g)=\sup_{i\in \mathbb Z}|\mathbb E_i(f)||\mathbb E_i(g)|,$$
respectively. Fix $i\in \mathbb Z,$ we define the tailed maximal operator
and tailed bilinear maximal operator by $${^*M}_if=\sup_{j\geq i}|\mathbb E_j(f)|\text{ and }
{^*\mathcal{M}_i}(f,g)=\sup_{j\geq i}|\mathbb E_j(f)||\mathbb E_j(g)|,$$
respectively.

Let $B\in \mathcal {F},$ $w\in \mathcal{L}^+,$
we always denote $\int_\Omega\chi_Bd\mu$ and
$\int_\Omega\chi_B\omega d\mu$ by $|B|$ and $|B|_\omega,$
respectively.

\subsection{bilinear weights}

In this subsection we define several kinds of bilinear weights.

\begin{definition}\label{revers}Let $\omega_1, ~\omega_2$ be weights and $1<p_1, ~p_2<\infty.$
Suppose that
$\frac{1}{p}=\frac{1}{p_1 }+\frac{1}{p_2 }.$
Denote that $\overrightarrow{p}=(p_1,p_2)$
and $\sigma_s=\omega_s^{-\frac{1}{p_s-1}}\in \mathcal{L}^+,~s=1,~2.$
We say that the couple of weights $(\omega_1, ~\omega_2)$
satisfies the reverse H\"{o}lder's condition $RH_{\overrightarrow{p}},$ if
there exists a positive constant $C$ such that for all $i\in \mathbb Z$ and $\tau\in \mathcal {T}_i,$ we have
\begin{equation}\label{RH}
\Big(\int_{\{\tau<+\infty\}}\sigma_1d\mu\Big)^{\frac{p}{p_1}}
\Big(\int_{\{\tau<+\infty\}}\sigma_2d\mu\Big)^{\frac{p}{p_2}}
\leq C\int_{\{\tau<+\infty\}}\sigma_1^{\frac{p}{p_1}}
\sigma_2^{\frac{p}{p_2}}d\mu.
\end{equation}
We denote by $[\omega_1, \omega_2]_{RH_{\overrightarrow{p}}}$
the smallest constant $C$ in \eqref{RH}.
\end{definition}

\begin{remark}
In the literature there exist many reverse H\"{o}lder's inequalities of the type
$$\|f\|_p\|g\|_q\leq C\|fg\|_1,$$
where $\frac{1}{p}+\frac{1}{q}=1,$ $C$ is a constant and the functions $f$ and $g$ are subjected to suitable
restrictions. The suitable restrictions can be found in \cite{Nehari,Zhuang}.
In our paper, we find that the reverse H\"{o}lder's condition is useful for bilinear weighted theory.
\end{remark}

\begin{definition}\label{definition Ap}Let $v,~\omega_1 \hbox{ and }\omega_2$ be weights and $1<p_1, ~p_2<\infty.$
Suppose that
$\frac{1}{p}=\frac{1}{p_1 }+\frac{1}{p_2 }.$
Denote that $\overrightarrow{p}=(p_1,p_2)$
and $\sigma_s=\omega_s^{-\frac{1}{p_s-1}}\in \mathcal{L}^+,~s=1,~2.$
We say that the triple of weights $(v,~\omega_1, ~\omega_2)$
satisfies the condition $A_{\overrightarrow{p}},$ if
there exists a positive constant $C$ such that
      \begin{equation}\label{bi-AP}
\sup\limits_{j\in \mathbb Z}\mathbb E_j(v)\mathbb E_j(\omega_1^{1-p'_1})^{\frac{p}{p'_1}}
      \mathbb E_j(\omega_2^{1-p'_2})^{\frac{p}{p'_2}}\leq C,
      \end{equation}
where $\frac{1}{p_s}+\frac{1}{p'_s}=1,~s=1,~2.$ We denote by $[v,\omega_1, \omega_2]_{A_{\overrightarrow{p}}}$
the smallest constant $C$ in \eqref{bi-AP}.
\end{definition}

\begin{definition}\label{definition Sp}Let $v,~\omega_1 \hbox{ and }\omega_2$ be weights and $1<p_1, ~p_2<\infty.$
Suppose that
$\frac{1}{p}=\frac{1}{p_1 }+\frac{1}{p_2 }.$
Denote that $\overrightarrow{p}=(p_1,p_2)$
and $\sigma_s=\omega_s^{-\frac{1}{p_s-1}}\in \mathcal{L}^+,~s=1,~2.$
We say that the triple of weights $(v,~\omega_1, ~\omega_2)$
satisfies the condition $S_{\overrightarrow{p}},$ if
      \begin{equation}\label{bi-SP}
[v,\overrightarrow{\omega}]_{S_{\overrightarrow{p}}}:=\sup\limits_{i\in \mathbb Z,\tau\in \mathcal {T}_i}
\Big(\frac{\int_{\{\tau<+\infty\}} {\mathcal{M}}(\sigma_1\chi_{\{\tau<+\infty\}},\sigma_2 \chi_{\{\tau<+\infty\}})^pvd\mu}{\sigma_1({\{\tau<+\infty\}})^{\frac{p}{p_1}}\sigma_2({\{\tau<+\infty\}})^{\frac{p}{p_2}}}\Big)^{\frac{1}{p}}<\infty,
      \end{equation}
where $\frac{1}{p_s}+\frac{1}{p'_s}=1,~s=1,~2.$
\end{definition}

\begin{definition}\label{definition Bp}Let $v,~\omega_1 \hbox{ and }\omega_2$ be weights and $1<p_1, ~p_2<\infty.$
Suppose that
$\frac{1}{p}=\frac{1}{p_1 }+\frac{1}{p_2 }.$
Denote that $\overrightarrow{p}=(p_1,p_2)$
and $\sigma_s=\omega_s^{-\frac{1}{p_s-1}}\in \mathcal{L}^+,~s=1,~2.$
We say that the couple of weights $(\omega_1, ~\omega_2)$
satisfies the condition $B_{\overrightarrow{p}},$ if
there exists a positive constant $C$ such that for all $i\in \mathbb Z$ we have
      \begin{equation}\label{bi-BP}\mathbb E_i(v)E_i(\sigma_1)^p
         \mathbb E_i(\sigma_2)^p\leq
C\exp\Big(\mathbb E_i(\log(\sigma^{\frac{p}{p_1}}_1\sigma^{\frac{p}{p_2}}_2))\Big).
      \end{equation}
We denote by $[v,\omega_1, \omega_2]_{B_{\overrightarrow{p}}}$
the smallest constant $C$ in \eqref{bi-BP}.
\end{definition}

\begin{definition}\label{definition Wp}Let $\omega_1 \hbox{ and }\omega_2$ be weights and $1<p_1, ~p_2<\infty.$
Suppose that
$\frac{1}{p}=\frac{1}{p_1 }+\frac{1}{p_2 }.$
Denote that $\overrightarrow{p}=(p_1,p_2)$
and $\sigma_s=\omega_s^{-\frac{1}{p_s-1}}\in \mathcal{L}^+,~s=1,~2.$
We say that the couple of weights $(\omega_1, ~\omega_2)$
satisfies the condition $W^{\infty}_{\overrightarrow{p}},$ if
there exists a positive constant $C$ such that for all $i\in \mathbb Z$ and $\tau\in \mathcal{T}_i$ we have
      \begin{equation}\label{bi-WP}\int_{\{\tau<+\infty\}}M(\sigma_1\chi_{\{\tau<+\infty\}})^{\frac{p}{p_1}}
         M(\sigma_2\chi_{\{\tau<+\infty\}})^{\frac{p}{p_2}}d\mu
\leq C\int_{\{\tau<+\infty\}} \sigma_1^{\frac{p}{p_1}}\sigma_2^{\frac{p}{p_2}}d \mu.
      \end{equation}
We denote by $[\omega_1, \omega_2]_{W^{\infty}_{\overrightarrow{p}}}$
the smallest constant $C$ in \eqref{bi-WP}.
\end{definition}

\begin{remark}
If $p_1=p_2$
and $\omega_1=\omega_2$ in the above definitions, we obtain the linear ones.
\end{remark}

\section{The construction of principal sets}\label{constru}
Let $i\in \mathbb Z,$ $h_1\in \mathcal{L}^+$ and $h_2\in \mathcal{L}^+.$
Fixing $k\in \mathbb Z,$ we define a stopping time$$
\tau:=\inf\{j\geq i:
~\mathbb E(h_1|\mathcal{F}_j)\mathbb E(h_2|\mathcal{F}_j)>4^{k+1}\}.$$
For $\Omega_0\in \mathcal{F}^0_i,$ we denote that \begin{equation}\label{P0}
P_0:=\{4^{k-1}< \mathbb E(h_1|\mathcal{F}_i)\mathbb E(h_2|\mathcal{F}_i)\leq4^k\}\cap\Omega_0,\end{equation}
and assume $\mu(P_0)>0.$
It follows that $P_0\in \mathcal{F}^0_i.$
We write $\mathcal{K}_1(P_0):=i$ and $\mathcal{K}_2(P_0):=k.$
We let $\mathcal{P}_1:=\{P_0\}$ which we call the first generation of principal sets.
To get the second generation of principal sets we define a stopping time
$$
\tau_{P_0}:=\tau\chi_{P_0}+\infty\chi_{P^c_0},$$
where $P^c_0=\Omega\setminus P_0.$
We say that a set $P\subset P_0$ is a principal set with respect to $P_0$
if it satisfies $\mu(P)>0$ and there exists $j>i$ and $l>k+1$ such that
\begin{eqnarray*}
P&=&\{4^{l-1}< \mathbb E(h_1|\mathcal{F}_j)
          \mathbb E(h_2|\mathcal{F}_j)\leq4^l\}\cap\{\tau_{P_0}=j\}\cap P_0\\
&=&\{4^{l-1}< \mathbb E(h_1|\mathcal{F}_j)
          \mathbb E(h_2|\mathcal{F}_j)\leq4^l\}\cap\{\tau=j\}\cap P_0.\end{eqnarray*}
Noticing that such $j$ and $l$ are unique, we write $\mathcal{K}_1(P):=j$ and $\mathcal{K}_2(P):=l.$
We let $\mathcal{P}(P_0)$ be the set of all principal sets with respect to $P_0$ and
let $\mathcal{P}_2:=\mathcal{P}(P_0)$ which we call the second generalization of principal sets.

We now need to verify that
$$
\mu(P_0)\leq2\mu\big(E(P_0)\big),
$$
where
$$
E(P_0):=P_0\cap\{\tau_{P_0}=\infty\}=P_0\cap\{\tau=\infty\}=P_0\backslash\bigcup\limits_{P\in \mathcal{P}(P_0)}P.
$$
Indeed, we have
\begin{eqnarray*}
& &\mu\big(P_0\cap\{\tau_{P_0}<\infty\}\big)\\
&\leq& (4^{-k-1})^\frac{1}{2}\int_{P_0\cap\{\tau_{P_0}<\infty\}}\mathbb E(h_1|\mathcal{F}_{\tau_{P_0}})^\frac{1}{2}
                 \mathbb E(h_2|\mathcal{F}_{\tau_{P_0}})^\frac{1}{2}d\mu\\
&=& (4^{-k-1})^\frac{1}{2}\int_{P_0}\mathbb E(h_1|\mathcal{F}_{\tau_{P_0}})^\frac{1}{2}
                 \mathbb E(h_2|\mathcal{F}_{\tau_{P_0}})^\frac{1}{2}\chi_{\{\tau_{P_0}<\infty\}}d\mu\\
&=& (4^{-k-1})^\frac{1}{2}\int_{P_0}\sum\limits_{j\geq i}\mathbb E(h_1|\mathcal{F}_{\tau_{P_0}})^\frac{1}{2}
                 \mathbb E(h_2|\mathcal{F}_{\tau_{P_0}})^\frac{1}{2}\chi_{\{\tau_{P_0}=j\}}d\mu\\
&=& (4^{-k-1})^\frac{1}{2}\int_{P_0}\sum\limits_{j\geq i}\mathbb E(h_1|\mathcal{F}_{j})^\frac{1}{2}
                 \mathbb E(h_2|\mathcal{F}_{j})^\frac{1}{2}\chi_{\{\tau_{P_0}=j\}}d\mu.\end{eqnarray*}
It follows from the H\"{o}lder's inequality for sum that
\begin{eqnarray*}
& &\mu\big(P_0\cap\{\tau_{P_0}<\infty\}\big)\\
&\leq& (4^{-k-1})^\frac{1}{2}\int_{P_0}\Big(\sum\limits_{j\geq i}\mathbb E(h_1\chi_{\{\tau_{P_0}=j\}}|\mathcal{F}_{j})\Big)^\frac{1}{2}
                 \Big(\sum\limits_{j\geq i}\mathbb E(h_2\chi_{\{\tau_{P_0}=j\}}|\mathcal{F}_{j})\Big)^\frac{1}{2}d\mu\\
&=& (4^{-k-1})^\frac{1}{2}\int_{P_0}\mathbb E_i\left(\Big(\sum\limits_{j\geq i}\mathbb E(h_1\chi_{\{\tau_{P_0}=j\}}|\mathcal{F}_{j})\Big)^\frac{1}{2}
                 \Big(\sum\limits_{j\geq i}\mathbb E(h_2\chi_{\{\tau_{P_0}=j\}}|\mathcal{F}_{j})\Big)^\frac{1}{2}\right)d\mu.\end{eqnarray*}
Applying the H\"{o}lder's inequality for conditional expectations, we have
\begin{eqnarray*}
& &\mu\big(P_0\cap\{\tau_{P_0}<\infty\}\big)\\
&\leq& (4^{-k-1})^\frac{1}{2}\int_{P_0}\Big(\sum\limits_{j\geq i}\mathbb E_i(\mathbb E(h_1\chi_{\{\tau_{P_0}=j\}}|\mathcal{F}_{j}))\Big)^\frac{1}{2}
\Big(\sum\limits_{j\geq i}\mathbb E_i(\mathbb E(h_2\chi_{\{\tau_{P_0}=j\}}|\mathcal{F}_{j}))\Big)^\frac{1}{2}d\mu\\
&=& (4^{-k-1})^\frac{1}{2}\int_{P_0}\Big(\sum\limits_{j\geq i}\mathbb E_i(h_1\chi_{\{\tau_{P_0}=j\}})\Big)^\frac{1}{2}
                \Big(\sum\limits_{j\geq i}\mathbb E_i(h_2\chi_{\{\tau_{P_0}=j\}})\Big)^\frac{1}{2}d\mu\\
&=& (4^{-k-1})^\frac{1}{2}\int_{P_0}\mathbb E_i(h_1\chi_{\{\tau_{P_0}<\infty\}})^\frac{1}{2}
                \mathbb E_i(h_2\chi_{\{\tau_{P_0}<\infty\}})^\frac{1}{2}d\mu\\
&\leq& (4^{-k-1})^\frac{1}{2}\int_{P_0}\mathbb E_i(h_1)^\frac{1}{2}
                \mathbb E_i(h_2)^\frac{1}{2}d\mu
\leq4^{-\frac{1}{2}}\mu(P_0)= \frac{1}{2}\mu(P_0).
\end{eqnarray*}
This clearly implies $$
\mu(P_0)\leq2\mu\big(E(P_0)\big).
$$

For any $P'_0\in (P_0\cap\mathcal{F}^0_i),$ there exists a set $\Omega''_0\in \mathcal{F}^0_i$ such that
$$P'_0=P_0\cap\Omega''_0=\{4^{k-1}< \mathbb E(h_1|\mathcal{F}_i)\mathbb E(h_2|\mathcal{F}_i)\leq4^k\}\cap\Omega_0\cap\Omega''_0.$$
Taking $\Omega'_0=\Omega_0\cap\Omega''_0,$ we have $P'_0=\{4^{k-1}< \mathbb E(h_1|\mathcal{F}_i)\mathbb E(h_2|\mathcal{F}_i)\leq4^k\}\cap\Omega'_0.$
Using $\Omega'_0$
instead of $\Omega_0$ in \eqref{P0}, we deduce that $$\mu(P'_0)\leq2\mu\big(E(P'_0)\big).$$
Moreover, we obtain that
\begin{eqnarray*}
\int_{P'_0}\chi_{P_0}d\mu&=&\mu(P'_0\cap P_0)=\mu(P'_0)
\leq2\mu\big(E(P'_0)\big)=2\mu\big(P'_0\cap\{\tau=\infty\}\big)\\
&=&2\mu\big(P'_0\cap P_0\cap\{\tau=\infty\}\big)
=2\int_{P'_0}\chi_{E(P_0)}d\mu\\
&=&2\int_{P'_0}\mathbb E_i(\chi_{E(P_0)})d\mu.
\end{eqnarray*}
Since $P'_0$ is arbitrary, we have $\chi_{P_0}\leq 2\mathbb E_i(\chi_{E(P_0)})\chi_{P_0}.$

The next generalizations are defined inductively,
$$
\mathcal{P}_{n+1}:=\bigcup\limits_{P\in \mathcal{P}_{n}}\mathcal{P}(P),
$$
and we define the collection of principal sets $\mathcal{P}$ by
$$
\mathcal{P}:=\bigcup\limits_{n=1}^{\infty}\mathcal{P}_{n}.
$$
It is easy to see that the collection of principal sets $\mathcal{P}$ satisfied the following properties:
\begin{itemize}
  \item [(P. 1)] The set $E(P)$ where $P\in \mathcal{P},$ are disjoint and $P_0=\bigcup\limits_{P\in \mathcal{P}}E(P);$
  \item [(P. 2)] $P\in\mathcal{F}_{{\mathcal{K}}_1(P)};$
  \item [(P. 3)] \label{new}$\chi_{P}\leq 2\mathbb E(\chi_{E(P)}|\mathcal{F}_{{\mathcal{K}}_1(P)})\chi_{P}$;
  \item [(P. 4)] $4^{{\mathcal{K}}_2(P)-1}<\mathbb E(h_1|\mathcal{F}_{{\mathcal{K}}_1(P)})
                \mathbb E(h_2|\mathcal{F}_{{\mathcal{K}}_1(P)})\leq4^{{\mathcal{K}}_2(P)}$ on $P;$
  \item [(P. 5)] $\sup\limits_{j\geq i}\mathbb E_j(h_1\chi_P)
                 \mathbb E_j(h_2\chi_P)\leq4^{{\mathcal{K}}_2(P)+1}$ on $E(P).$
\end{itemize}

We call (P.\ref{new}) the conditional sparsity of principal sets.
Then we use the principal sets to represent the tailed bilinear maximal operator and obtain the following Lemma $\ref{repre}.$

\begin{lemma}\label{repre}Let $i\in \mathbb{Z},$ $h_1\in \mathcal{L}^+$ and $h_2\in \mathcal{L}^+.$
Fixing $k\in \mathbb{Z}$ and $\Omega_0\in \mathcal{F}^0_i,$ we denote $$
P_0:=\{4^{k-1}< \mathbb E(h_1|\mathcal{F}_i)\mathbb E(h_2|\mathcal{F}_i)\leq4^k\}\cap\Omega_0.$$
If $\mu(P_0)>0,$ then
\begin{eqnarray*}
{^*\mathcal{M}_i}(h_1,h_2)\chi_{P_0}
&=&{^*\mathcal{M}_i}(h_1\chi_{P_0},h_2\chi_{P_0})\chi_{P_0}\\
&=&\sum\limits_{P\in \mathcal{P}}{^*\mathcal{M}_i}(h_1\chi_{P_0},h_2\chi_{P_0})\chi_{E(P)}\\
&\leq&16\sum\limits_{P\in \mathcal{P}}4^{({\mathcal{K}}_2(P)-1)}\chi_{E(P)}.
\end{eqnarray*}
\end{lemma}

\section{Carleson embedding theorem associated with the collection of principal sets}\label{Car}

For $\omega_1\in \mathcal{L}^+$ and $\omega_2\in \mathcal{L}^+,$ we set $\sigma_1:=\omega_1^{-\frac{1}{p_1-1}}\in \mathcal{L}^+$ and $\sigma_2=:\omega_2^{-\frac{1}{p_2-1}}\in \mathcal{L}^+.$ Suppose that $f^{p_1}_1\omega_1\in L_1^+$ and $f^{p_2}_2\omega_2\in L_1^+.$
It follows from H\"{o}lder's inequality that $f_1\in \mathcal{L}^+$ and $f_2\in \mathcal{L}^+.$ Let $h_1=f_1$ and $h_2=f_2.$
Fixing $k\in \mathbb Z,$ $i\in \mathbb Z$ and $\Omega_0\in \mathcal{F}^0_i$ such that $
\mu(\{4^{k-1}< \mathbb E(h_1|\mathcal{F}_i)\mathbb E(h_2|\mathcal{F}_i)\leq4^k\}\cap\Omega_0)>0,$ we apply the construction of principal sets to give the following Carleson embedding theorem.

\begin{theorem}\label{Carleson_thm}For $P\in \mathcal{P}$ and $l\in \mathbb Z,$ let \begin{equation}\label{con1}A_P^l:=P\cap\{2^l<\mathbb E(\sigma_1|\mathcal{F}_{\mathcal{K}_1(P)})
\mathbb E(\sigma_2|\mathcal{F}_{\mathcal{K}_1(P)})\leq2^{l+1}\}.\end{equation} We denote
$\mathcal{Q}=:\bigcup\limits_{P\in \mathcal{P}}\bigcup\limits_{l\in \mathbb Z}A_P^l.$
If the nonnegative numbers $a_Q$ and non-negative
function $\sigma_1^{\frac{p}{p_1}}\sigma_2^{\frac{p}{p_2}}$
satisfy
\begin{equation}\label{lem_Carleson_ass}
\sum\limits_{Q\subseteq \{\tau<+\infty\}}a_Q\leq A \int_{\{\tau<+\infty\}} \sigma_1^{\frac{p}{p_1}}\sigma_2^{\frac{p}{p_2}}d \mu,
~\forall Q\in \mathcal{Q},~\tau\in \mathcal{T}_i,\end{equation}
where $A$ is an absolute constant,
then
\begin{multline*}
\sum\limits_{A_P^l\in \mathcal{Q}}\mathop{\hbox{essinf}}\limits_{A_P^l}
   \Big(\mathbb E^{\sigma_1}(h_1\sigma^{-1}_1|\mathcal{F}_{\mathcal{K}_1(P)})\mathbb E^{\sigma_2}(h_2\sigma^{-1}_2|\mathcal{F}_{\mathcal{K}_1(P)})\Big)^pa_{A_P^l}\\
\leq A(p'_1p'_2)^p\Big(\int_{P_0}h_1^{p_1}\omega_1 d\mu\Big)^{\frac{p}{p_1}}\Big(\int_{P_0}h_2^{p_2}\omega_2 d\mu\Big)^{\frac{p}{p_2}},
\end{multline*}
where $\mathbb E^{\sigma_s}(\cdot|\mathcal{F}_{\mathcal{K}_1(P)})$ is the conditional expectation with respect to $\mathcal{F}_{\mathcal{K}_1(P)},$
$\sigma_s d\mu$ in place of $d\mu,$ $s=1,~2.$
\end{theorem}

\begin{proof} We view the sum $$\sum\limits_{A_P^l\in \mathcal{Q}}\mathop{\hbox{essinf}}\limits_{A_P^l}
   \Big(\mathbb E^{\sigma_1}(h_1\sigma^{-1}_1|\mathcal{F}_{\mathcal{K}_1(P)})\mathbb E^{\sigma_2}
   (h_2\sigma^{-1}_{2}|\mathcal{F}_{\mathcal{K}_1(P)})\Big)^pa_{A_P^l}$$
as an integral on a measure space $(\mathcal{Q},~2^\mathcal{Q},~\nu)$ built over
$\mathcal{Q},$ assigning to each $Q\in\mathcal{Q}$ the measure $a_Q.$ Thus
$$\sum\limits_{A_P^l\in \mathcal{Q}}\mathop{\hbox{essinf}}\limits_{A_P^l}
   \Big(\mathbb E^{\sigma_1}(h_1\sigma^{-1}_1|\mathcal{F}_{\mathcal{K}_1(P)})\mathbb E^{\sigma_2}(h_2\sigma^{-1}_2|\mathcal{F}_{\mathcal{K}_1(P)})\Big)^pa_{A_P^l}\\
=\int_0^\infty p\lambda^{p-1}\nu\big(D_\lambda)d\lambda,$$
where $\mathcal{D}_\lambda=\big\{A_P^l\in \mathcal{Q}:~\mathop{\hbox{essinf}}\limits_{A_P^l}
   \big(\mathbb E^{\sigma_1}(h_1\sigma^{-1}_1|\mathcal{F}_{\mathcal{K}_1(P)})\mathbb E^{\sigma_2}(h_2\sigma^{-1}_2|\mathcal{F}_{\mathcal{K}_1(P)})\big)>\lambda\big\}.$
Let $$\tau=\inf\Big\{n\geq i:~\mathbb E^{\sigma_1}(h_1\sigma^{-1}_1|\mathcal{F}_n)\mathbb E^{\sigma_2}(h_2\sigma^{-1}_2|\mathcal{F}_n)\chi_{P_0}>\lambda\Big\}.$$
Then $\tau\in \mathcal{T}_i$ and $\bigcup\limits_{Q\in \mathcal{D}_\lambda}Q\subset\{{^*\mathcal{M}}_i^{\sigma_1,\sigma_2}(h_1\sigma^{-1}_1\chi_{P_0},h_2\sigma^{-1}_2\chi_{P_0})>\lambda\}=\{\tau< +\infty\},$
where $${^*\mathcal{M}}_i^{\sigma_1,\sigma_2}(h_1\sigma^{-1}_1\chi_{P_0},h_2\sigma^{-1}_2\chi_{P_0}):=\sup\limits_{j\geq i}\mathbb E^{\sigma_1}(h_1\sigma^{-1}_1\chi_{P_0}|\mathcal{F}_j)\mathbb E^{\sigma_2}(h_2\sigma^{-1}_2\chi_{P_0}|\mathcal{F}_j).$$
Thus
\begin{eqnarray*}
\nu(\mathcal{D}_\lambda)=\sum\limits_{Q\in \mathcal{D}_\lambda}a_Q
&\leq& A\int_{\{{^*\mathcal{M}}_i^{\sigma_1,\sigma_2}(h_1\sigma^{-1}_1\chi_{P_0},\,h_2\sigma^{-1}_2\chi_{P_0})>\lambda\}}
       \sigma_1^{\frac{p}{p_1}}\sigma_2^{\frac{p}{p_2}}d \mu\\
&\leq&A\int_{\{\mathcal{M}^{\sigma_1,\sigma_2}(h_1\sigma^{-1}_1\chi_{P_0},\,h_2\sigma^{-1}_2\chi_{P_0})>\lambda\}}
       \sigma_1^{\frac{p}{p_1}}\sigma_2^{\frac{p}{p_2}}d \mu,
\end{eqnarray*}
which implies that
\begin{multline}\label{2}
\sum\limits_{A_P^l\in \mathcal{Q}}\mathop{\hbox{essinf}}\limits_{A_P^l}
   \Big(\mathbb E^{\sigma_1}(h_1\sigma^{-1}_1|\mathcal{F}_{\mathcal{K}_1(P)})\mathbb E^{\sigma_2}(h_2\sigma^{-1}_2|\mathcal{F}_{\mathcal{K}_1(P)})\Big)^pa_{A_P^l}\\
\leq A\int_0^\infty p\lambda^{p-1}\int_{\{\mathcal{M}^{\sigma_1,\sigma_2}(h_1\sigma^{-1}_1\chi_{P_0},\,h_2\sigma^{-1}_2\chi_{P_0})>\lambda\}}
       \sigma_1^{\frac{p}{p_1}}\sigma_2^{\frac{p}{p_2}}d \mu,
\end{multline}
where $${\mathcal{M}}^{\sigma_1,\sigma_2}(h_1\sigma^{-1}_1\chi_{P_0},h_2\sigma^{-1}_2\chi_{P_0}):=\sup\limits_{j\in \mathbb{Z}}\mathbb E^{\sigma_1}(h_1\sigma^{-1}_1\chi_{P_0}|\mathcal{F}_j)\mathbb E^{\sigma_2}(h_2\sigma^{-1}_2\chi_{P_0}|\mathcal{F}_j).$$
It follows from Fubini's theorem that
\begin{multline}\label{3}
\int_0^\infty p\lambda^{p-1}\int_{\{\mathcal{M}^{\sigma_1,\sigma_2}(h_1\sigma^{-1}_1\chi_{P_0},h_2\sigma^{-1}_2\chi_{P_0})>\lambda\}}
       \sigma_1^{\frac{p}{p_1}}\sigma_2^{\frac{p}{p_2}}d \mu\\
=\int_{\Omega}\mathcal{M}^{\sigma_1,\sigma_2}(h_1\sigma^{-1}_1\chi_{P_0},h_2\sigma^{-1}_2\chi_{P_0})^p \sigma_1^{\frac{p}{p_1}}\sigma_2^{\frac{p}{p_2}}d\mu.
\end{multline}
Applying H\"{o}lder's inequality and Doob's maximal inequality, we obtain that
\begin{eqnarray*}
&&\int_{\Omega}\mathcal{M}^{\sigma_1,\sigma_2}(h_1\sigma^{-1}_1\chi_{P_0},h_2\sigma^{-1}_2\chi_{P_0})^p \sigma_1^{\frac{p}{p_1}}\sigma_2^{\frac{p}{p_2}}d\mu
\\&\leq&\int_{\Omega}\Big(\big(M^{\sigma_1}(h_1\sigma^{-1}_1\chi_{P_0})\big)^{p_1}\sigma_1\Big)^{\frac{p}{p_1}}
       \Big(\big(M^{\sigma_2}(h_2\sigma^{-1}_2\chi_{P_0})\big)^{p_2}\sigma_2\Big)^{\frac{p}{p_2}} d\mu\\
&\leq&\Big(\int_{\Omega}\big(M^{\sigma_1}(h_1\sigma^{-1}_1\chi_{P_0})\big)^{p_1}\sigma_1d\mu\Big)^{\frac{p}{p_1}}
       \Big(\int_{\Omega}\big(M^{\sigma_2}(h_2\sigma^{-1}_2\chi_{P_0})\big)^{p_2}\sigma_2d\mu\Big)^{\frac{p}{p_2}} \\
&\leq&(p'_1p'_2)^p\Big(\int_{P_0}h^{p_1}_1\sigma^{1-p_1}_1 d\mu\Big)^{\frac{p}{p_1}}\Big(\int_{P_0}h^{p_1}_2\sigma^{1-p_2}_2 d\mu\Big)^{\frac{p}{p_2}}\\
&=&(p'_1p'_2)^p\Big(\int_{P_0}h_1^{p_1}\omega_1 d\mu\Big)^{\frac{p}{p_1}}\Big(\int_{P_0}h_2^{p_2}\omega_2 d\mu\Big)^{\frac{p}{p_2}},
\end{eqnarray*}
where $$M^{\sigma_s}(h_s\sigma^{-1}_s\chi_{P_0}):=\sup\limits_{j\in \mathbb{Z}}\mathbb E^{\sigma_s}(h_s\sigma^{-1}_s\chi_{P_0}|\mathcal{F}_j),~s=1,2.$$
Combining \eqref{2}, \eqref{3} and the inequalities above, we conclude this proof.
\end{proof}

\begin{remark} \label{re_ca} Using $A_P^l:=E(P)\cap\{2^l<\mathbb E(\sigma_1|\mathcal{F}_{\mathcal{K}_1(P)})
\mathbb E(\sigma_2|\mathcal{F}_{\mathcal{K}_1(P)})\leq2^{l+1}\}$ instead of $\eqref{con1},$ we still have Theorem \ref{lem_Carleson_ass}.
\end{remark}

\section{Main results and their proofs }\label{proofs}

\subsection{Bilinear Version of one-weight Inequalities}

\begin{proof}[\bf Proof of Theorem \ref{thm_AP}.] \eqref{re1} Let $i\in \mathbb Z$ be arbitrarily chosen and fixed.
For $k\in \mathbb Z$ and $\Omega_0\in \mathcal{F}^0_i,$ we denote $$
P_0=\{4^{k-1}< \mathbb E(f_1\sigma_1|\mathcal{F}_i)\mathbb E(f_2\sigma_2|\mathcal{F}_i)\leq4^k\}\cap\Omega_0.$$
We claim that
\begin{multline}\label{6}
\Big(\int_{P_0}{^*\mathcal{M}_i}
(f_1\sigma_1\chi_{P_0},f_2\sigma_2\chi_{P_0})^{p}vd\mu\Big)^{\frac{1}{p}}\\ \leq
16\cdot4^{(q'-1)}p'_1p'_2[v,\omega_1,\omega_2]^{\frac{q'}{p}}_{A_{\overrightarrow P}}\Big(\int_{P_0} f_1^{p_1}\sigma_1d\mu\Big)^{\frac{1}{p_1}}\Big(\int_{P_0} f_2^{p_2}\sigma_2d\mu\Big)^{\frac{1}{p_2}},
\end{multline}
 where $q=\min\{p_1,p_2\}.$
To see this, denote $h_1=f_1\sigma_1\chi_{P_0}$ and $h_2=f_2\sigma_2\chi_{P_0}.$
For the above $i,$ $P_0,$ $h_1$ and $h_2,$ we apply the construction of principal sets. It follows from Lemma $\ref{repre}$ that
\begin{equation}\label{5}
\int_{P_0}{^*\mathcal{M}_i}
(f_1\sigma_1,f_2\sigma_2)^{p}vd\mu
\leq16^p\sum\limits_{P\in \mathcal{P}}\int_{E(P)}4^{p({\mathcal{K}}_2(P)-1)}vd\mu.
\end{equation}
Without loss of generality assume that $1< p_1\leq p_2<\infty.$  We now estimate $\int_{E(P)}vd\mu$ as follows:

\begin{eqnarray*}
\int_{E(P)}vd\mu&\leq&\int_{P}vd\mu=\int_{P}\mathbb E(v|\mathcal{F}_{{\mathcal{K}}_1(P)})d\mu\\
&=&\int_{P}\mathbb E(v|\mathcal{F}_{{\mathcal{K}}_1(P)})^{p'_1}\mathbb E(v|\mathcal{F}_{{\mathcal{K}}_1(P)})^{1-p'_1}
\mathbb E(\sigma_1|\mathcal{F}_{{\mathcal{K}}_1(P)})^{\frac{pp'_1}{p'_1}}\\
&&\times\mathbb E(\sigma_2|\mathcal{F}_{{\mathcal{K}}_1(P)})^{\frac{pp'_1}{p'_2}} \mathbb E(\sigma_1|\mathcal{F}_{{\mathcal{K}}_1(P)})^{-\frac{pp'_1}{p'_1}}\mathbb E(\sigma_2|\mathcal{F}_{{\mathcal{K}}_1(P)})^{-\frac{pp'_1}{p'_2}}d\mu\\
&=&\int_{P}\mathbb E(v|\mathcal{F}_{{\mathcal{K}}_1(P)})^{p'_1}
\mathbb E(\sigma_1|\mathcal{F}_{{\mathcal{K}}_1(P)})^{\frac{pp'_1}{p'_1}}\mathbb E(\sigma_2|\mathcal{F}_{{\mathcal{K}}_1(P)})^{\frac{pp'_1}{p'_2}}\\
&&\times \mathbb E(v|\mathcal{F}_{{\mathcal{K}}_1(P)})^{1-p'_1}
\mathbb E(\sigma_1|\mathcal{F}_{{\mathcal{K}}_1(P)})^{-\frac{pp'_1}{p'_1}}\mathbb E(\sigma_2|\mathcal{F}_{{\mathcal{K}}_1(P)})^{-\frac{pp'_1}{p'_2}}d\mu.
\end{eqnarray*}
It follows from the definition of $A_{\overrightarrow P}$ and the conditional sparsity of principal sets that
\begin{eqnarray*}
&&\int_{E(P)}vd\mu\\
&\leq&[v,\omega_1,\omega_2]^{p_1'}_{A_{\overrightarrow P}}\int_{P}\mathbb E(v|\mathcal{F}_{{\mathcal{K}}_1(P)})^{1-p'_1}
\mathbb E(\sigma_1|\mathcal{F}_{{\mathcal{K}}_1(P)})^{-\frac{pp'_1}{p'_1}}\mathbb E(\sigma_2|\mathcal{F}_{{\mathcal{K}}_1(P)})^{-\frac{pp'_1}{p'_2}}d\mu\\
&\leq&2^{2p(p'_1-1)}[v,\omega_1,\omega_2]^{p_1'}_{A_{\overrightarrow P}}\int_{P}\mathbb E(v|\mathcal{F}_{{\mathcal{K}}_1(P)})^{1-p'_1}
    \mathbb E(\sigma_1|\mathcal{F}_{{\mathcal{K}}_1(P)})^{-\frac{pp'_1}{p'_1}}\\
&&\times  \mathbb E(\sigma_2|\mathcal{F}_{{\mathcal{K}}_1(P)})^{-\frac{pp'_1}{p'_2}}
    \mathbb E(\chi_{E(P)}|\mathcal{F}_{{\mathcal{K}}_1(P)})^{2p(p'_1-1)}d\mu\\
&=&2^{2p(p'_1-1)}[v,\omega_1,\omega_2]^{p_1'}_{A_{\overrightarrow P}}\int_{P}\mathbb E(v|\mathcal{F}_{{\mathcal{K}}_1(P)})^{1-p'_1}
    \mathbb E(\sigma_1|\mathcal{F}_{{\mathcal{K}}_1(P)})^{-\frac{pp'_1}{p'_1}}\\
&&\times \mathbb E(\sigma_2|\mathcal{F}_{{\mathcal{K}}_1(P)})^{-\frac{pp'_1}{p'_2}} \mathbb E(\chi_{E(P)}v^{\frac{1}{2p}}\sigma_1^{\frac{1}{2p'_1}}\sigma_2^{\frac{1}{2p'_2}}|\mathcal{F}_{{\mathcal{K}}_1(P)})^{2p(p'_1-1)}d\mu.
\end{eqnarray*}
Applying H\"{o}lder's inequality for the conditional expectation, we have
\begin{eqnarray*} \int_{E(P)}vd\mu
&\leq&2^{2p(p'_1-1)}[v,\omega_1,\omega_2]^{p_1'}_{A_{\overrightarrow P}}\int_{P}\mathbb E(v|\mathcal{F}_{{\mathcal{K}}_1(P)})^{1-p'_1}
\mathbb E(\sigma_1|\mathcal{F}_{{\mathcal{K}}_1(P)})^{-\frac{pp'_1}{p'_1}}\\
&&\times\mathbb E(\sigma_2|\mathcal{F}_{{\mathcal{K}}_1(P)})^{-\frac{pp'_1}{p'_2}}  \Big(\mathbb E(v \chi_{E(P)}|\mathcal{F}_{{\mathcal{K}}_1(P)})^{\frac{1}{2p}}
\mathbb E(\sigma_1\chi_{E(P)}|\mathcal{F}_{{\mathcal{K}}_1(P)})^{\frac{1}{2p'_1}}\\
&&\times
\mathbb E(\sigma_2\chi_{E(P)}|\mathcal{F}_{{\mathcal{K}}_1(P)})^{\frac{1}{2p'_2}}\Big)^{2p(p'_1-1)}d\mu.\end{eqnarray*}
It follows from $p'_1\geq p'_2$ that $\frac{1}{2p'_s}2p(p'_1-1)-\frac{p}{p_s}=\frac{pp'_1}{p'_s}-p\geq 0,~s=1,~2.$ Then
\begin{eqnarray*}
\mathbb E(\sigma_s\chi_{E(P)}|\mathcal{F}_{{\mathcal{K}}_1(P)})^{\frac{1}{2p'_s}2p(p'_1-1)-\frac{p}{p_s}}
&=&\mathbb E(\sigma_s\chi_{E(P)}|\mathcal{F}_{{\mathcal{K}}_1(P)})^{\frac{pp'_1}{p'_s}-p}\\
&\leq& \mathbb E(\sigma_s|\mathcal{F}_{{\mathcal{K}}_1(P)})^{\frac{pp'_1}{p'_s}-p},~s=1,~2.
\end{eqnarray*}
Thus
\begin{eqnarray*}
\int_{E(P)}vd\mu
&\leq&2^{2p(p'_1-1)}[v,\omega_1,\omega_2]^{p_1'}_{A_{\overrightarrow P}}\int_{P}\mathbb E(v\chi_{E(P)}|\mathcal{F}_{{\mathcal{K}}_1(P)})^{1-p'_1} \mathbb E(\sigma_1|\mathcal{F}_{{\mathcal{K}}_1(P)})^{-\frac{pp'_1}{p'_1}}\\
&&\times
\mathbb E(\sigma_2|\mathcal{F}_{{\mathcal{K}}_1(P)})^{-\frac{pp'_1}{p'_2}} \Big(\mathbb E(v\chi_{E(P)}|\mathcal{F}_{{\mathcal{K}}_1(P)})^{\frac{1}{2p}}
\mathbb E(\sigma_1\chi_{E(P)}|\mathcal{F}_{{\mathcal{K}}_1(P)})^{\frac{1}{2p'_1}}\\
&&\times
\mathbb E(\sigma_2\chi_{E(P)}|\mathcal{F}_{{\mathcal{K}}_1(P)})^{\frac{1}{2p'_2}}\Big)^{2p(p'_1-1)}d\mu\\
&\leq&2^{2p(p'_1-1)}[v,\omega_1,\omega_2]^{p_1'}_{A_{\overrightarrow P}}\int_{P}\mathbb E(\sigma_1|\mathcal{F}_{{\mathcal{K}}_1(P)})^{-p}
\mathbb E(\sigma_2|\mathcal{F}_{{\mathcal{K}}_1(P)})^{-p}\\
&&\times \mathbb E(\sigma_1\chi_{E(P)}|\mathcal{F}_{{\mathcal{K}}_1(P)})^{\frac{p}{p_1}}
\mathbb E(\sigma_2\chi_{E(P)}|\mathcal{F}_{{\mathcal{K}}_1(P)})^{\frac{p}{p_2}}d\mu.
\end{eqnarray*}
Noting that $E(P)\subset P$ and $4^{{\mathcal{K}}_2(P)-1}<\mathbb E(h_1|\mathcal{F}_{{\mathcal{K}}_1(P)})
                \mathbb E(h_2|\mathcal{F}_{{\mathcal{K}}_1(P)})$ on $P,$ we obtain that
\begin{eqnarray*}
\int_{E(P)}4^{p({\mathcal{K}}_2(P)-1)}vd\mu
&\leq&2^{2p(p'_1-1)}[v,\omega_1,\omega_2]^{p_1'}_{A_{\overrightarrow P}}\int_{P}\Big(\mathbb E(f_1\sigma_1|\mathcal{F}_{{\mathcal{K}}_1(P)})
                \mathbb E(f_2\sigma_2|\mathcal{F}_{{\mathcal{K}}_1(P)})\Big)^p\\
&&\times \mathbb E(\sigma_1|\mathcal{F}_{{\mathcal{K}}_1(P)})^{-p}
\mathbb E(\sigma_2|\mathcal{F}_{{\mathcal{K}}_1(P)})^{-p}\\
&&\times \mathbb E(\chi_{E(P)}\sigma_1|\mathcal{F}_{{\mathcal{K}}_1(P)})^{\frac{p}{p_1}}
     \mathbb E(\chi_{E(P)}\sigma_2|\mathcal{F}_{{\mathcal{K}}_1(P)})^{\frac{p}{p_2}}d\mu\\
&=&2^{2p(p'_1-1)}[v,\omega_1,\omega_2]^{p_1'}_{A_{\overrightarrow P}}\int_{P}\Big(\mathbb E^{\sigma_1}(f_1|\mathcal{F}_{{\mathcal{K}}_1(P)})
                 \mathbb E^{\sigma_2}(f_2|\mathcal{F}_{{\mathcal{K}}_1(P)})\Big)^p\\
&&\times \mathbb E(\chi_{E(P)}\sigma_1|\mathcal{F}_{{\mathcal{K}}_1(P)})^{\frac{p}{p_1}}
      \mathbb E(\chi_{E(P)}\sigma_2|\mathcal{F}_{{\mathcal{K}}_1(P)})^{\frac{p}{p_2}}d\mu,
\end{eqnarray*}
where the last equality uses a standard fact that $$ \mathbb E(f\sigma|\mathcal{F}_{{\mathcal{K}}_1(P)})=\mathbb E^\sigma(f|\mathcal{F}_{{\mathcal{K}}_1(P)})\mathbb E(\sigma|\mathcal{F}_{{\mathcal{K}}_1(P)}).$$
Using H\"{o}lder's inequality, we get
\begin{eqnarray*}
&&\int_{E(P)}4^{p({\mathcal{K}}_2(P)-1)}vd\mu\\
&\leq&2^{2p(p'_1-1)}[v,\omega_1,\omega_2]^{p_1'}_{A_{\overrightarrow P}}\Big(\int_{P}\mathbb E^{\sigma_1}(f_1|\mathcal{F}_{{\mathcal{K}}_1(P)})^{p_1}
\mathbb E(\chi_{E(P)}\sigma_1|\mathcal{F}_{{\mathcal{K}}_1(P)})d\mu\Big)^{\frac{p}{p_1}}\\
&&\times \Big(\int_{P}\mathbb E^{\sigma_2}(f_2|\mathcal{F}_{{\mathcal{K}}_1(P)})^{p_2}
\mathbb E(\chi_{E(P)}\sigma_2|\mathcal{F}_{{\mathcal{K}}_1(P)})d\mu\Big)^{\frac{p}{p_2}}\\
&=&2^{2p(p'_1-1)}[v,\omega_1,\omega_2]^{p_1'}_{A_{\overrightarrow P}}\Big(\int_{P}\mathbb E^{\sigma_1}(f_1|\mathcal{F}_{{\mathcal{K}}_1(P)})^{p_1}
\chi_{E(P)}\sigma_1d\mu\Big)^{\frac{p}{p_1}}\\
&&\times \Big(\int_{P}\mathbb E^{\sigma_2}(f_2|\mathcal{F}_{{\mathcal{K}}_1(P)})^{p_2}
\chi_{E(P)}\sigma_2d\mu\Big)^{\frac{p}{p_2}}\\
&\leq&2^{2p(p'_1-1)}[v,\omega_1,\omega_2]^{p_1'}_{A_{\overrightarrow P}}\Big(\int_{P}M^{\sigma_1}(f_1\chi_{P_0})^{p_1}
\chi_{E(P)}\sigma_1d\mu\Big)^{\frac{p}{p_1}}\\
&&\times \Big(\int_{P}M^{\sigma_2}(f_2\chi_{P_0})^{p_2}
\chi_{E(P)}\sigma_2d\mu\Big)^{\frac{p}{p_2}}\\
&=&2^{2p(p'_1-1)}[v,\omega_1,\omega_2]^{p_1'}_{A_{\overrightarrow P}}\Big(\int_{E(P)}M^{\sigma_1}(f_1\chi_{P_0})^{p_1}
\sigma_1d\mu\Big)^{\frac{p}{p_1}} \Big(\int_{E(P)}M^{\sigma_2}(f_2\chi_{P_0})^{p_2}\sigma_2d\mu\Big)^{\frac{p}{p_2}}.
\end{eqnarray*}
It follows from \eqref{5} that
\begin{eqnarray*}
&&\int_{P_0}{^*\mathcal{M}_i}
(f_1\sigma_1,f_1\sigma_2)^{p}vd\mu\\
&\leq&16^p2^{2p(p'_1-1)}[v,\omega_1,\omega_2]^{p_1'}_{A_{\overrightarrow P}}\sum\limits_{P\in \mathcal{P}}\Big(\int_{E(P)}M^{\sigma_1}(f_1\chi_{P_0})^{p_1}
\sigma_1d\mu\Big)^{\frac{p}{p_1}}\\
&&\times \Big(\int_{E(P)}M^{\sigma_2}(f_2\chi_{P_0})^{p_2}\sigma_2d\mu\Big)^{\frac{p}{p_2}}.
\end{eqnarray*}
Applying the H\"{o}lder's inequality for sum and
Doob's maximal inequality, we obtain that
\begin{eqnarray*}
&&\int_{P_0}{^*\mathcal{M}_i}
(f_1\sigma_1,f_1\sigma_2)^{p}vd\mu\\
&\leq&16^p2^{2p(p'_1-1)}[v,\omega_1,\omega_2]^{p_1'}_{A_{\overrightarrow P}}\big(\sum\limits_{P\in \mathcal{P}}\int_{E(P)}M^{\sigma_1}(f_1\chi_{P_0})^{p_1}
\sigma_1d\mu\big)^{\frac{p}{p_1}}\\
&&\times \Big(\sum\limits_{P\in \mathcal{P}}\int_{E(P)}M^{\sigma_2}(f_2\chi_{P_0})^{p_2}\sigma_2d\mu\Big)^{\frac{p}{p_2}}\\
&\leq&16^p2^{2p(p'_1-1)}[v,\omega_1,\omega_2]^{p_1'}_{A_{\overrightarrow P}}\Big(\int_{P_0}M^{\sigma_1}(f_1\chi_{P_0})^{p_1}
\sigma_1d\mu\Big)^{\frac{p}{p_1}}\\
&&\times \Big(\int_{P_0}M^{\sigma_2}(f_2\chi_{P_0})^{p_2}\sigma_2d\mu\Big)^{\frac{p}{p_2}}\\
&\leq&16^p2^{2p(p'_1-1)}(p'_1)^{p}(p'_2)^p[v,\omega_1,\omega_2]^{p_1'}_{A_{\overrightarrow P}}\Big(\int_{P_0}f^{p_1}_1
\sigma_1d\mu\Big)^{\frac{p}{p_1}} \Big(\int_{P_0}f_2^{p_2}\sigma_2d\mu\Big)^{\frac{p}{p_2}}.
\end{eqnarray*}

Hence, the estimation \eqref{6} is proved. Consequently,

\begin{eqnarray*}
& &\big(\int_{\Omega_0}{^*\mathcal{M}_i}
(f_1\sigma_1,f_2\sigma_2)^{p}vd\mu\big)^{\frac{1}{p}}\\
&=&\big(\sum\limits_{k\in Z}\int_{\{4^{k-1}< E(f_1\sigma_1|\mathcal{F}_i)E(f_2\sigma_2|\mathcal{F}_i)\leq4^k\}\cap\Omega_0}{^*\mathcal{M}_i}
(f_1\sigma_1,f_2\sigma_2)^{p}vd\mu\big)^{\frac{1}{p}}\\
&\leq&16\cdot4^{(p'_1-1)}p'_1p'_2[v,\omega_1,\omega_2]^{\frac{p_1'}{p}}_{A_{\overrightarrow P}}\\
&&\times\Big(\sum\limits_{k\in Z}\big(\int_{\{4^{k-1}< E(f_1\sigma_1|\mathcal{F}_i)E(f_2\sigma_2|\mathcal{F}_i)\leq4^k\}\cap\Omega_0}f^{p_1}_1
\sigma_1d\mu\big)^{\frac{p}{p_1}}\\
&&\times \big(\int_{\{4^{k-1}< E(f_1\sigma_1|\mathcal{F}_i)E(f_2\sigma_2|\mathcal{F}_i)\leq4^k\}\cap\Omega_0}f_2^{p_2}\sigma_2d\mu\big)^{\frac{p}{p_2}}\Big)^{\frac{1}{p}}.
\end{eqnarray*}

It follows from the H\"{o}lder's inequality for sum that

\begin{eqnarray*}
& &\big(\int_{\Omega_0}{^*\mathcal{M}_i}
(f_1\sigma_1,f_2\sigma_2)^{p}vd\mu\big)^{\frac{1}{p}}\\
&\leq&16\cdot4^{(p'_1-1)}p'_1p'_2[v,\omega_1,\omega_2]^{\frac{p_1'}{p}}_{A_{\overrightarrow P}}\\
&&\times\big(\sum\limits_{k\in Z}\int_{\{4^{k-1}< E(f_1\sigma_1|\mathcal{F}_i)E(f_2\sigma_2|\mathcal{F}_i)\leq4^k\}\cap\Omega_0}f^{p_1}_1
\sigma_1d\mu\big)^{\frac{1}{p_1}}\\
&&\times \big(\sum\limits_{k\in Z}\int_{\{4^{k-1}< E(f_1\sigma_1|\mathcal{F}_i)E(f_2\sigma_2|\mathcal{F}_i)\leq4^k\}\cap\Omega_0}f_2^{p_2}\sigma_2d\mu\big)^{\frac{1}{p_2}}\\
&\leq&16\cdot4^{(p'_1-1)}p'_1p'_2[v,\omega_1,\omega_2]^{\frac{p_1'}{p}}_{A_{\overrightarrow P}}
\big(\int_{\Omega_0}f^{p_1}_1
\sigma_1d\mu\big)^{\frac{1}{p_1}}\big(\int_{\Omega_0}f_2^{p_2}\sigma_2d\mu\big)^{\frac{1}{p_2}}.
\end{eqnarray*}

Since the measure space $(\Omega,\mathcal{F},\mu)$ is $\sigma$-finite, we have
\begin{eqnarray*}
&&\Big(\int_{\Omega}{^*{\mathcal{M}}_i}
(f_1\sigma_1,f_2\sigma_2)^{p}vd\mu\Big)^{\frac{1}{p}}\\&&\leq16\cdot4^{(p'_1-1)}p'_1p'_2[v,\omega_1,\omega_2]^{\frac{p_1'}{p}}_{A_{\overrightarrow P}}
\Big(\int_{\Omega}f^{p_1}_1
\sigma_1d\mu\Big)^{\frac{1}{p_1}}\Big(\int_{\Omega}f_2^{p_2}\sigma_2d\mu\Big)^{\frac{1}{p_2}}.
\end{eqnarray*}
Using the monotone convergence theorem, we obtain that
\begin{eqnarray*}
& &\Big(\int_{\Omega}{\mathcal{M}}
(f_1\sigma_1,f_2\sigma_2)^{p}vd\mu\Big)^{\frac{1}{p}}\\
&\leq&16\cdot4^{(p'_1-1)}p'_1p'_2[v,\omega_1,\omega_2]^{\frac{p_1'}{p}}_{A_{\overrightarrow P}}
\Big(\int_{\Omega}f^{p_1}_1
\sigma_1d\mu\Big)^{\frac{1}{p_1}}\Big(\int_{\Omega}f_2^{p_2}\sigma_2d\mu\Big)^{\frac{1}{p_2}}.
\end{eqnarray*}

\eqref{re2}
Fix $i\in \mathbb Z.$ For $B\in \mathcal{F}^0_i,$ set
$f_1=\chi_B$ and $f_2=\chi_B.$
Then $$\mathbb E_i(\omega_1^{-\frac{1}{p_1-1}})
\mathbb E_i(\omega_2^{-\frac{1}{p_2-1}})\chi_B\leq\mathcal{M}(f_1\sigma_1,f_2\sigma_2)\chi_B.$$
It follows from the assumption that
\begin{eqnarray*}& &\Big(\int_B\mathbb E_i(\omega_1^{-\frac{1}{p_1-1}})^p
\mathbb E_i(\omega_2^{-\frac{1}{p_2-1}})^pvd\mu\Big)^{\frac{1}{p}}\\
      &\leq& \|\mathcal{M}\|\Big(\int_\Omega \omega_1^{-\frac{1}{p_1-1}}\chi_Bd\mu\Big)^{\frac{1}{p_1}}
            \Big(\int_\Omega \omega_2^{-\frac{1}{p_2-1}}\chi_Bd\mu\Big)^{\frac{1}{p_2}}.\end{eqnarray*}
Since $(\omega_1, \omega_2)\in RH_{\overrightarrow{p}},$ we have
\begin{eqnarray*}& &\int_B\mathbb E_i(\omega_1^{-\frac{1}{p_1-1}})^p
\mathbb E_i(\omega_2^{-\frac{1}{p_2-1}})^pvd\mu\\
      &\leq& \|\mathcal{M}\|^p[\omega_1, \omega_2]_{RH_{\overrightarrow{p}}}\Big(\int_B \omega_1^{-\frac{1}{p_1-1}\frac{p}{p_1}}
            \omega_2^{-\frac{1}{p_2-1}\frac{p}{p_2}}d\mu\Big).\end{eqnarray*}
Thus
\begin{eqnarray*}
&&\mathbb E_i(\omega_1^{-\frac{1}{p_1-1}})^p
            \mathbb E_i(\omega_2^{-\frac{1}{p_2-1}})^p\mathbb E_i(v)\\
&\leq& \|\mathcal{M}\|^p[\omega_1, \omega_2]_{RH_{\overrightarrow{p}}}\mathbb E_i(\omega_1^{-\frac{1}{p_1-1}\frac{p}{p_1}}
            \omega_2^{-\frac{1}{p_2-1}\frac{p}{p_2}})\\
            &\leq& \|\mathcal{M}\|^p[\omega_1, \omega_2]_{RH_{\overrightarrow{p}}}\mathbb E_i(\omega_1^{-\frac{1}{p_1-1}})^{\frac{p}{p_1}}
           \mathbb E_i(\omega_2^{-\frac{1}{p_2-1}})^{\frac{p}{p_2}},\end{eqnarray*}
where we have used H\"{o}lder's inequality for conditional expectations.
Then we obtain $$\mathbb E_i(v)^{\frac{1}{p}}\mathbb E_i(\omega_1^{1-p'_1})^{\frac{1}{p'_1}}
      \mathbb E_i(\omega_2^{1-p'_2})^{\frac{1}{p'_2}}\leq \|\mathcal{M}\|[\omega_1, \omega_2]^{\frac{1}{p}}_{RH_{\overrightarrow{p}}}.$$
\end{proof}

\begin{corollary} \label{cor_AP}Let $\omega_1, ~\omega_2$ be weights and $1< p_1, ~p_2<\infty.$
Suppose that
$\frac{1}{p}=\frac{1}{p_1 }+\frac{1}{p_2 }$ and $v=\omega_1^{\frac{p}{p_1}}\omega_2^{\frac{p}{p_2}}.$
\begin{enumerate}[\rm(1)]
\item  \label{re1} If $(v,\omega_1, \omega_2)\in A_{\overrightarrow{p}}$, then there exists a positive constant $C$ such that
for all $f_1\in L^{p_1}(\omega_1),~f_2\in L^{p_2}(\omega_2)$ we have
      \begin{equation}\label{Th_DB_2}\|\mathcal{M}(f_1,f_2)\|_{L^{p}(v)}\leq
       C\|f_1\|_{L^{p_1}(\omega_1)}\|f_2\|_{L^{p_2}(\omega_2)}.
      \end{equation}
We denote the smallest constant $C$ in \eqref{Th_DB_2} by $\|\mathcal{M}\|.$
Then it follows that
$$\|\mathcal{M}\|\leq16\cdot4^{(q'-1)}p'_1p'_2[v,\omega_1,\omega_2]^{\frac{q'}{p}}_{A_{\overrightarrow P}},$$ where $q=\min\{p_1,p_2\}.$

\item  \label{re2} Let $(\omega_1, \omega_2)\in RH_{\overrightarrow{p}}.$ If there exists a positive constant $C$ such that
for all $f_1\in L^{p_1}(\omega_1),~f_2\in L^{p_2}(\omega_2)$ we have
      \begin{equation}\label{Th_BC_2}\|\mathcal{M}(f_1,f_2)\|_{L^{p}(v)}\leq
       C\|f_1\|_{L^{p_1}(\omega_1)}\|f_2\|_{L^{p_2}(\omega_2)},
      \end{equation}
then $(v,~\omega_1, ~\omega_2)\in A_{\overrightarrow{p}}.$ We denote the smallest constant $C$ in \eqref{Th_BC_2} by $\|\mathcal{M}\|.$
Then it follows that
$[v,\omega_1,\omega_2]_{A_{\overrightarrow P}}\leq\|\mathcal{M}\|^{p}[\omega_1, \omega_2]_{RH_{\overrightarrow{p}}}.$
\end{enumerate}
\end{corollary}

\subsection{Bilinear Version of Two-weight Inequalities}

\begin{proof}[\bf Proof of Theorem \ref{thm_Sp}. ]
To prove $\eqref{Spa}\Rightarrow \eqref{Spb}.$ Fix $i\in \mathbb Z.$ For $\tau\in \mathcal{T}_i,$ set
$f_1=\chi_{\{\tau<+\infty\}}$ and $f_2=\chi_{\{\tau<+\infty\}}.$ It follows from \eqref{thmSp} that $$\|\mathcal{M}(\sigma_1\chi_{\{\tau<+\infty\}},\sigma_2\chi_{\{\tau<+\infty\}})\|_{L^p(v)}\leq
C\|\chi_{\{\tau<+\infty\}}\|_{L^{p_1}(\sigma_1)}\|\chi_{\{\tau<+\infty\}}\|_{L^{p_2}(\sigma_2)}.$$ Thus $[v,\overrightarrow{\omega}]_{S_{\overrightarrow{p}}}
\leq\|\mathcal{M}\|.$

To prove $\eqref{Spb}\Rightarrow\eqref{Spa},$ as we show in Theorem \ref{thm_AP}, we have
\begin{eqnarray}\label{7}
\int_{P_0}{^*\mathcal{M}_i}(f_1\sigma_1,f_2\sigma_2)^{p}vd\mu
\leq16^p\sum\limits_{P\in \mathcal{P}}\int_{E(P)}4^{p({\mathcal{K}}_2(P)-1)}vd\mu.
\end{eqnarray}
This also implies that
\begin{eqnarray*}
\int_{P_0}{^*\mathcal{M}_i}(f_1\sigma_1,f_2\sigma_2)^{p}vd\mu
\leq16^p\sum\limits_{P\in \mathcal{P}}\sum\limits_{l\in \mathbb Z}\int_{A^l_P}4^{p({\mathcal{K}}_2(P)-1)}vd\mu,
\end{eqnarray*}
where $A^l_P:=E(P)\cap\{2^l<\mathbb E(\sigma_1|\mathcal{F}_{\mathcal{K}_1(P)})\mathbb E(\sigma_2|\mathcal{F}_{\mathcal{K}_1(P)})\leq2^{l+1}\}.$
By the properties of principal sets, we have
\begin{eqnarray*}
4^{({\mathcal{K}}_2(P)-1)}\chi_{A^l_P}
&\leq&\mathbb E(h_1|\mathcal{F}_{{\mathcal{K}}_1(P)})
                \mathbb E(h_2|\mathcal{F}_{{\mathcal{K}}_1(P)})\chi_{A^l_P}\\
&=&\mathbb E^{\sigma_1}(h_1\sigma^{-1}_1|\mathcal{F}_{\mathcal{K}_1(P)})
        \mathbb E^{\sigma_2}(h_2\sigma^{-1}_2|\mathcal{F}_{\mathcal{K}_1(P)})\mathbb E(\sigma_1|\mathcal{F}_{\mathcal{K}_1(P)})
         \mathbb E(\sigma_2|\mathcal{F}_{\mathcal{K}_1(P)})\chi_{A^l_P}\\
&\leq&\mathbb E^{\sigma_1}(h_1\sigma^{-1}_1|\mathcal{F}_{\mathcal{K}_1(P)})
         \mathbb E^{\sigma_2}(h_2\sigma^{-1}_2|\mathcal{F}_{\mathcal{K}_1(P)}) 2^{l+1}\chi_{A^l_P}.\end{eqnarray*}
It follows that
\begin{eqnarray*} 4^{({\mathcal{K}}_2(P)-1)}\chi_{A^l_P}
&\leq&\mathop{\hbox{essinf}}\limits_{A_P^l}\Big(\mathbb E^{\sigma_1}(h_1\sigma^{-1}_1|\mathcal{F}_{\mathcal{K}_1(P)})
         \mathbb E^{\sigma_2}(h_2\sigma^{-1}_2|\mathcal{F}_{\mathcal{K}_1(P)})\Big)2^{l+1}\chi_{A^l_P}\\
&\leq&2\mathop{\hbox{essinf}}\limits_{A_P^l}\Big(\mathbb E^{\sigma_1}(h_1\sigma^{-1}_1|\mathcal{F}_{\mathcal{K}_1(P)})
         \mathbb E^{\sigma_2}(h_2\sigma^{-1}_2|\mathcal{F}_{\mathcal{K}_1(P)})\Big)\\
& &\times \mathbb E(\sigma_1|\mathcal{F}_{\mathcal{K}_1(P)})
         \mathbb E(\sigma_2|\mathcal{F}_{\mathcal{K}_1(P)})\chi_{A^l_P}.
\end{eqnarray*}
For simplicity, we denote  $$a_{A^l_P}:=\int_{A^l_P}\big(\mathbb E(\sigma_1|\mathcal{F}_{\mathcal{K}_1(P)})
         \mathbb E(\sigma_2|\mathcal{F}_{\mathcal{K}_1(P)})\big)^pvd\mu.$$

Then
\begin{eqnarray*}
&&\int_{P_0}{^*\mathcal{M}_i}(f_1\sigma_1,f_2\sigma_2)^{p}vd\mu
\leq16^p\sum\limits_{P\in \mathcal{P}}\sum\limits_{l\in \mathbb Z}\int_{A^l_P}4^{p({\mathcal{K}}_2(P)-1)}vd\mu\\
&\leq&32^p\sum\limits_{P\in \mathcal{P}}\sum\limits_{l\in \mathbb Z}\mathop{\hbox{essinf}}\limits_{A_P^l}
   \Big(\mathbb E^{\sigma_1}(h_1\sigma^{-1}_1|\mathcal{F}_{\mathcal{K}_1(P)})E^{\sigma_2}(h_2\sigma^{-1}_2|\mathcal{F}_{\mathcal{K}_1(P)})\Big)^p\\
& &\times\int_{A^l_P}\Big(\mathbb E(\sigma_1|\mathcal{F}_{\mathcal{K}_1(P)})
         \mathbb E(\sigma_2|\mathcal{F}_{\mathcal{K}_1(P)})\Big)^pvd\mu\\
&=&32^p\sum\limits_{P\in \mathcal{P}}\sum\limits_{l\in \mathbb Z}\mathop{\hbox{essinf}}\limits_{A_P^l}
   \Big(\mathbb E^{\sigma_1}(h_1\sigma^{-1}_1|\mathcal{F}_{\mathcal{K}_1(P)})\mathbb E^{\sigma_2}(h_2\sigma^{-1}_2|\mathcal{F}_{\mathcal{K}_1(P)})\Big)^p
   a_{A^l_P}.
\end{eqnarray*}
Now we claim that
\begin{multline}\label{8}
\Big(\int_{P_0}{^*\mathcal{M}_i}
(f_1\sigma_1\chi_{P_0},f_2\sigma_2\chi_{P_0})^{p}vd\mu\Big)^{\frac{1}{p}}\\
\leq 32p'_1p'_2
[v,\overrightarrow{\omega}]_{S_{\overrightarrow{p}}}[\omega_1, \omega_2]^{\frac{1}{p}}_{RH_{\overrightarrow{p}}}\Big(\int_{P_0} f_1^{p_1}\sigma_1d\mu\Big)^{\frac{1}{p_1}}\Big(\int_{P_0} f_2^{p_2}\sigma_2d\mu\Big)^{\frac{1}{p_2}}.
\end{multline}
To see this, we apply the
Carleson embedding theorem to these $a_{A^l_P}.$ By Theorem \ref{Carleson_thm}, it suffices to prove
\begin{equation}\label{9}
\sum\limits_{A^l_P\subseteq \{\tau<+\infty\}}a_{A^l_P}\leq A \int_{\{\tau<+\infty\}} \sigma_1^{\frac{p}{p_1}}\sigma_2^{\frac{p}{p_2}}d \mu,
~\tau\in \mathcal{T}_i.
\end{equation}
For $\tau\in \mathcal{T}_i,$ we have
\begin{eqnarray*}
&&\sum\limits_{A^l_P\subseteq \{\tau<+\infty\}}a_{A^l_P}
=\sum\limits_{A^l_P\subseteq \{\tau<+\infty\}}
\int_{A^l_P}\Big(\mathbb E(\sigma_1|\mathcal{F}_{\mathcal{K}_1(P)})
         \mathbb E(\sigma_2|\mathcal{F}_{\mathcal{K}_1(P)})\Big)^pvd\mu\\
&&=\sum\limits_{A^l_P\subseteq \{\tau<+\infty\}}
     \int_{A^l_P}\mathbb E(\sigma_1\chi_{\{\tau<+\infty\}}|\mathcal{F}_{\mathcal{K}_1(P)})^p
         \mathbb E(\sigma_2\chi_{\{\tau<+\infty\}}|\mathcal{F}_{\mathcal{K}_1(P)})^pvd\mu\\
&&\leq\sum\limits_{A^l_P\subseteq \{\tau<+\infty\}}
     \int_{A^l_P}\mathcal{M}(\sigma_1\chi_{\{\tau<+\infty\}},\sigma_2\chi_{\{\tau<+\infty\}})^pvd\mu\\
&&\leq\int_{\{\tau<+\infty\}}\mathcal{M}(\sigma_1\chi_{\{\tau<+\infty\}},\sigma_2\chi_{\{\tau<+\infty\}})^pvd\mu\\
&&\leq [v,\overrightarrow{\omega}]^p_{S_{\overrightarrow{p}}} \Big(\int_{\{\tau<+\infty\}} \sigma_1d \mu\Big)^{\frac{p}{p_1}} \Big(\int_{\{\tau<+\infty\}} \sigma_2d \mu\Big)^{\frac{p}{p_2}}\\
&&\leq[v,\overrightarrow{\omega}]^p_{S_{\overrightarrow{p}}}[\omega_1, \omega_2]_{RH_{\overrightarrow{p}}}\int_{\{\tau<+\infty\}} \sigma_1^{\frac{p}{p_1}}\sigma_2^{\frac{p}{p_2}}d \mu.\end{eqnarray*}
Therefore, \eqref{9} is proved and \eqref{8} immediately follows.
Employing an argument similar to the one in the proof of Theorem \ref{thm_AP}, we have
\begin{eqnarray*}
& &\Big(\int_{\Omega}{\mathcal{M}}
(f_1\sigma_1,f_2\sigma_2)^{p}vd\mu\Big)^{\frac{1}{p}}\\
&&\leq 32p'_1p'_2
[v,\overrightarrow{\omega}]_{S_{\overrightarrow{p}}}[\omega_1, \omega_2]^{\frac{1}{p}}_{RH_{\overrightarrow{p}}}\Big(\int_{\Omega} f_1\sigma_1d\mu\Big)^{\frac{1}{p_1}}\Big(\int_{\Omega} f_2\sigma_2d\mu\Big)^{\frac{1}{p_2}}.
\end{eqnarray*}
\end{proof}

\begin{corollary}\label{cor_Sp} Let $v,\omega_1, \omega_2$ be weights and $1<p_1, p_2<\infty.$
Suppose that
$1/p=1/p_1+1/p_2$ and $(\omega_1, \omega_2)\in RH_{\overrightarrow{p}},$
then the following statements
are equivalent$:$
\begin{enumerate}[$\rm(1)$]
                 \item There exists a positive constant $C$ such that
\begin{equation}\label{thm_1Sp}
\|\mathcal{M}(f_1,f_2)\|_{L^p(v)}\leq
C\|f_1\|_{L^{p_1}(\omega_1)}\|f_2\|_{L^{p_2}(\omega_2)},
~\forall f\in L^{p_1}(\omega_1),~g\in L^{p_2}(\omega_2);
\end{equation}
                 \item The triple of weights $(v,~\omega_1, ~\omega_2)$
satisfies the condition $S_{\overrightarrow{p}}.$
               \end{enumerate}
Moreover, we denote the smallest constant $C$ in \eqref{thm_1Sp}
by $\|\mathcal{M}\|.$
Then it follows that
$$[v,\overrightarrow{\omega}]_{S_{\overrightarrow{p}}}
\leq\|\mathcal{M}\|\leq32p'_1p'_2 [v,\overrightarrow{\omega}]_{S_{\overrightarrow{p}}}
[\overrightarrow{\omega}]_{RH_{\overrightarrow{p}}}^{\frac{1}{p}}.$$
\end{corollary}

\begin{corollary} \label{cor_APSP}Let $\omega_1, ~\omega_2$ be weights and $1< p_1, ~p_2<\infty.$ Suppose that
$\frac{1}{p}=\frac{1}{p_1 }+\frac{1}{p_2 }$ and $v=\omega_1^{\frac{p}{p_1}}\omega_2^{\frac{p}{p_2}}.$ If $(v,\omega_1, \omega_2)\in A_{\overrightarrow{p}}$, then $(v,\omega_1, \omega_2)\in S_{\overrightarrow{p}}$ and $[v,\overrightarrow{\omega}]_{S_{\overrightarrow{p}}}
\leq 16\cdot4^{(q'-1)}p'_1p'_2[v,\omega_1,\omega_2]^{\frac{q'}{p}}_{A_{\overrightarrow P}}.$
\end{corollary}

\begin{proof}[\bf Proof of Theorem \ref{theorem_Bp}.] It is clear that $\eqref{thm B 1}\Leftrightarrow\eqref{thm B 2},$
so $\|\mathcal{M}\|=\|\mathcal{M}\|'.$  To prove $\eqref{thm B 2},$ noting that \eqref{5} and $E(P)\subset P$, we have
\begin{eqnarray*}
\int_{P_0}{^*\mathcal{M}_i}(f_1\sigma_1,f_2\sigma_2)^{p}vd\mu
\leq16^p\sum\limits_{P\in \mathcal{P}}\int_{P}4^{p({\mathcal{K}}_2(P)-1)}vd\mu.
\end{eqnarray*}
It follows that
\begin{eqnarray*}
\int_{P_0}{^*\mathcal{M}_i}(f_1\sigma_1,f_2\sigma_2)^{p}vd\mu
\leq16^p\sum\limits_{P\in \mathcal{P}}\sum\limits_{l\in \mathbb Z}\int_{A_P^l}4^{p({\mathcal{K}}_2(P)-1)}vd\mu,
\end{eqnarray*}
where $A^l_P=P\cap\{2^l<\mathbb E(\sigma_1|\mathcal{F}_{\mathcal{K}_1(P)})\mathbb E(\sigma_2|\mathcal{F}_{\mathcal{K}_1(P)})\leq2^{l+1}\}.$
Following the arguments used in the proof of Theorem \ref{thm_Sp}, we denote
$$a_{A^l_P}=\int_{A^l_P}\Big(\mathbb E(\sigma_1|\mathcal{F}_{\mathcal{K}_1(P)})\mathbb E(\sigma_2|\mathcal{F}_{\mathcal{K}_1(P)})\Big)^pvd\mu.$$
Then
\begin{multline}\label{10}
\int_{P_0}{^*\mathcal{M}_i}(f_1\sigma_1,f_2\sigma_2)^{p}vd\mu
\\ \leq 32^p\sum\limits_{P\in \mathcal{P}}\sum\limits_{l\in \mathbb Z}\mathop{\hbox{essinf}}\limits_{A_P^l}
   \Big(\mathbb E^{\sigma_1}(h_1\sigma^{-1}_1|\mathcal{F}_{\mathcal{K}_1(P)})\mathbb E^{\sigma_2}(h_2\sigma^{-1}_2|\mathcal{F}_{\mathcal{K}_1(P)})\Big)^p
   a_{A^l_P}.
\end{multline}
Applying the
Carleson embedding theorem to these $a_{A^l_P},$ we claim that
\begin{equation}\label{11}
\sum\limits_{A^l_P\subseteq \{\tau<+\infty\}}a_{A^l_P}\leq 2e [v,\omega_1, \omega_2]_{B_{\overrightarrow{p}}}\int_{\{\tau<+\infty\}} \sigma_1^{\frac{p}{p_1}}\sigma_2^{\frac{p}{p_2}}d \mu,~\tau\in \mathcal{T}_i.
\end{equation}
In fact, for $\tau\in \mathcal{T}_i,$ noting that $A_P^l\subset P,$ we have
\begin{eqnarray*}
& &\sum\limits_{A^l_P\subseteq \{\tau<+\infty\}}a_{A^l_P}
=\sum\limits_{A^l_P\subseteq \{\tau<+\infty\}}
\int_{A^l_P}\mathbb E(\sigma_1|\mathcal{F}_{\mathcal{K}_1(P)})^p
        \mathbb E(\sigma_2|\mathcal{F}_{\mathcal{K}_1(P)})^pvd\mu\\
&=&\sum\limits_{A^l_P\subseteq \{\tau<+\infty\}}
     \int_{A^l_P}\mathbb E(\sigma_1|\mathcal{F}_{\mathcal{K}_1(P)})^p
         \mathbb E(\sigma_2|\mathcal{F}_{\mathcal{K}_1(P)})^p\mathbb E(v|\mathcal{F}_{\mathcal{K}_1(P)})d\mu\\
&\leq&[v,\omega_1, \omega_2]_{B_{\overrightarrow{p}}}\sum\limits_{A^l_P\subseteq \{\tau<+\infty\}}
     \int_{A^l_P}\exp\Big(\mathbb E(\log(\sigma^{\frac{p}{p_1}}_1\sigma^{\frac{p}{p_2}}_2)|\mathcal{F}_{\mathcal{K}_1(P)})\Big)d\mu\\
&=&[v,\omega_1, \omega_2]_{B_{\overrightarrow{p}}}\sum\limits_{A^l_P\subseteq \{\tau<+\infty\}}
     \int_{A^l_P}\exp\Big(\mathbb E(\log(\sigma^{\frac{p}{p_1}}_1\sigma^{\frac{p}{p_2}}_2\chi_{A^l_P})|\mathcal{F}_{\mathcal{K}_1(P)})\Big)\chi_{P}d\mu
\\&=:&[v,\omega_1, \omega_2]_{B_{\overrightarrow{p}}}\cdot {\rm I}.\end{eqnarray*}
It follows from the conditional sparsity of principal sets that ${\rm I}$ is controlled by
\begin{eqnarray*}
 2\sum\limits_{A^l_P\subseteq \{\tau<+\infty\}}
\int_{A^l_P}\exp\Big(\mathbb E(\log(\sigma^{\frac{p}{p_1}}_1\sigma^{\frac{p}{p_2}}_2\chi_{A^l_P})|\mathcal{F}_{\mathcal{K}_1(P)})\Big)
\mathbb E(\chi_{E(P)}|\mathcal{F}_{\mathcal{K}_1(P)})d\mu,
\end{eqnarray*}
which is equal to
$$2\sum\limits_{A^l_P\subseteq \{\tau<+\infty\}}
     \int_{A^l_P}\Big(\exp\Big(\mathbb E(\log(\sigma^{\frac{p}{rp_1}}_1\sigma^{\frac{p}{rp_2}}_2\chi_{A^l_P})|\mathcal{F}_{\mathcal{K}_1(P)})\Big)\Big)^r
     \mathbb E(\chi_{E(P)}|\mathcal{F}_{\mathcal{K}_1(P)})d\mu,
$$
where $r$ is an arbitrary real number and bigger than $1.$ Using Jensen's inequality, we have
$$\exp\Big(\mathbb E(\log(\sigma^{\frac{p}{rp_1}}_1\sigma^{\frac{p}{rp_2}}_2\chi_{A^l_P})|\mathcal{F}_{\mathcal{K}_1(P)})\Big)\leq
\mathbb E(\sigma^{\frac{p}{rp_1}}_1\sigma^{\frac{p}{rp_2}}_2\chi_{A^l_P}|\mathcal{F}_{\mathcal{K}_1(P)}).$$
Since for all $P\in \mathcal P(P)$ and $l\in \mathbb Z,$ $\mathbb E(P)\cap A_p^l$ are disjoint sets, it follows that
\begin{eqnarray*}
{\rm I}&\leq&2\sum\limits_{A^l_P\subseteq \{\tau<+\infty\}}
     \int_{A^l_P}\Big(\mathbb E(\sigma^{\frac{p}{rp_1}}_1\sigma^{\frac{p}{rp_2}}_2\chi_{A^l_P}|\mathcal{F}_{\mathcal{K}_1(P)})\Big)^r
    \mathbb E(\chi_{E(P)}|\mathcal{F}_{\mathcal{K}_1(P)})d\mu \\
&=&2\sum\limits_{A^l_P\subseteq \{\tau<+\infty\}}
     \int_{A^l_P}\Big(\mathbb E(\sigma^{\frac{p}{rp_1}}_1\sigma^{\frac{p}{rp_2}}_2\chi_{A^l_P}|\mathcal{F}_{\mathcal{K}_1(P)})\Big)^r
   \chi_{E(P)}d\mu\\
&\leq&2\sum\limits_{A^l_P\subseteq \{\tau<+\infty\}}
     \int_{E(P)\cap A_P^l}M\big(\sigma^{\frac{p}{rp_1}}_1\sigma^{\frac{p}{rp_2}}_2\chi_{\{\tau<+\infty\}}\big)^rd\mu\\
&\leq&
     2\int_{\Omega}M\big(\sigma^{\frac{p}{rp_1}}_1\sigma^{\frac{p}{rp_2}}_2\chi_{\{\tau<+\infty\}}\big)^rd\mu
     \leq 2\Big(\frac{r}{r-1}\Big)^r\int_{\{\tau<+\infty\}}\sigma^{\frac{p}{p_1}}_1\sigma^{\frac{p}{p_2}}_2d\mu.\end{eqnarray*}
Letting $r\rightarrow+\infty,$ we deduce that
\begin{eqnarray*}
{\rm I}\leq 2e\int_{\{\tau<+\infty\}} \sigma_1^{\frac{p}{p_1}}\sigma_2^{\frac{p}{p_2}}d \mu.\end{eqnarray*}
Therefore, the estimation \eqref{11} is proved. It follows from Theorem \ref{Carleson_thm} and \eqref{10} that
\begin{multline*}
\Big(\int_{P_0}{^*\mathcal{M}_i}
(f_1\sigma_1\chi_{P_0},f_2\sigma_2\chi_{P_0})^{p}vd\mu\Big)^{\frac{1}{p}}\\
\leq
32(2e)^{\frac{1}{p}}p'_1p'_2[v,\omega_1, \omega_2]^{\frac{1}{p}}_{B_{\overrightarrow{p}}}\Big(\int_{P_0} f_1^{p_1}\sigma_1d\mu\Big)^{\frac{1}{p_1}}\Big(\int_{P_0} f_2^{p_2}\sigma_2d\mu\Big)^{\frac{1}{p_2}}.
\end{multline*}
Then by similar arguments as Theorem \ref{thm_AP}, we get
\begin{eqnarray*}\Big(\int_\Omega\mathcal{M}(f_1\sigma_1,f_2\sigma_2)^p v d\mu\Big)^{\frac{1}{p}}
\leq32(2e)^{\frac{1}{p}}p'_1p'_2[v,\omega_1, \omega_2]^{\frac{1}{p}}_{B_{\overrightarrow{p}}}\|f_1\|_{L^{p_1}(\sigma_1)}\|f_2\|_{L^{p_2}(\sigma_2)}.\end{eqnarray*}
This completes the proof.
\end{proof}

\begin{proof}[\bf Proof of Theorem \ref{theorem_bi A Fujii}.]  This proof is similar to the proof of Theorem \ref{theorem_Bp}. For $A_P^l$ and $a_{A_P^l}$ defined in the proof of Theorem \ref{theorem_Bp}, it suffices to check the Carleson embedding condition,
$$\sum\limits_{A_P^l\subseteq \{\tau<+\infty\}}a_{A_P^l}\leq 2[v,\omega_1, \omega_2]_{A_{\overrightarrow{p}}}[\omega_1, \omega_2]_{W^{\infty}_{\overrightarrow{p}}} \int_{\{\tau<+\infty\}} \sigma_1^{\frac{p}{p_1}}\sigma_2^{\frac{p}{p_2}}d \mu,~\tau\in \mathcal{T}_i.$$
Indeed, for $\tau\in \mathcal{T}_i,$ it follows from the definition of $A_{\overrightarrow{p}}$ that
\begin{eqnarray*}
&&\sum\limits_{A_P^l\subseteq \{\tau<+\infty\}}a_{A_P^l}\\
&=&\sum\limits_{A^l_P\subseteq \{\tau<+\infty\}}
     \int_{A^l_P}\mathbb E(\sigma_1|\mathcal{F}_{\mathcal{K}_1(P)})^p
         \mathbb E(\sigma_2|\mathcal{F}_{\mathcal{K}_1(P)})^p\mathbb E(v|\mathcal{F}_{\mathcal{K}_1(P)})d\mu \\
&\leq&[v,\omega_1, \omega_2]_{A_{\overrightarrow{p}}}\sum\limits_{A^l_P\subseteq \{\tau<+\infty\}}
     \int_{A^l_P}\mathbb E(\sigma_1|\mathcal{F}_{\mathcal{K}_1(P)})^{\frac{p}{p_1}}
         \mathbb E(\sigma_2|\mathcal{F}_{\mathcal{K}_1(P)})^{\frac{p}{p_2}}d\mu\\
&=&[v,\omega_1, \omega_2]_{A_{\overrightarrow{p}}}\sum\limits_{A^l_P\subseteq \{\tau<+\infty\}}
     \int_{A^l_P}\mathbb E(\sigma_1\chi_{A^l_P}|\mathcal{F}_{\mathcal{K}_1(P)})^{\frac{p}{p_1}}
        \mathbb E(\sigma_2\chi_{A^l_P}|\mathcal{F}_{\mathcal{K}_1(P)})^{\frac{p}{p_2}}\chi_{P}d\mu\\
&=:&[v,\omega_1, \omega_2]_{A_{\overrightarrow{p}}}\cdot{\rm II}.
\end{eqnarray*}
It follows from the conditional sparsity of principal sets that ${\rm II}$ is controlled by
\begin{eqnarray*}
2\sum\limits_{A^l_P\subseteq \{\tau<+\infty\}}
     \int_{A^l_P}\mathbb E(\sigma_1\chi_{A^l_P}|\mathcal{F}_{\mathcal{K}_1(P)})^{\frac{p}{p_1}}
         \mathbb E(\sigma_2\chi_{A^l_P}|\mathcal{F}_{\mathcal{K}_1(P)})^{\frac{p}{p_2}}\mathbb E(\chi_{E(P)}|\mathcal{F}_{\mathcal{K}_1(P)})d\mu,
\end{eqnarray*}
which is smaller than
$$2\sum\limits_{A^l_P\subseteq \{\tau<+\infty\}}
     \int_{E(P)\cap A_P^l} M(\sigma_1\chi_{A^l_P})^{\frac{p}{p_1}}
         M(\sigma_2\chi_{A^l_P})^{\frac{p}{p_2}}d\mu.$$
It follows from the definition of $W^{\infty}_{\overrightarrow{p}}$ that
\begin{eqnarray*}
{\rm II}&\leq&2\int_{\{\tau<+\infty\}}M(\sigma_1\chi_{\{\tau<+\infty\}})^{\frac{p}{p_1}}
         M(\sigma_2\chi_{\{\tau<+\infty\}})^{\frac{p}{p_1}}d\mu\\
&\leq&2[v,\omega_1, \omega_2]_{A_{\overrightarrow{p}}}[\omega_1, \omega_2]_{W^{\infty}_{\overrightarrow{p}}}\int_{\{\tau<+\infty\}} \sigma_1^{\frac{p}{p_1}}\sigma_2^{\frac{p}{p_2}}d \mu.\end{eqnarray*}
Therefore, by \eqref{10} and Theorem \ref{Carleson_thm}, we obtain that
\begin{multline*}
\Big(\int_{P_0}{^*\mathcal{M}_i}
(f_1\sigma_1\chi_E,f_2\sigma_2\chi_E)^{p}vd\mu\Big)^{\frac{1}{p}}\\
\leq32
\cdot2^{\frac{1}{p}}p'_1p'_2[v,\omega_1, \omega_2]^{\frac{1}{p}}_{A_{\overrightarrow{p}}}[\omega_1, \omega_2]^{\frac{1}{p}}_{W^{\infty}_{\overrightarrow{p}}}\Big(\int_{P_0} f_1^{p_1}\sigma_1d\mu\Big)^{\frac{1}{p_1}}\Big(\int_{P_0} f_2^{p_2}\sigma_2d\mu\Big)^{\frac{1}{p_2}}.
\end{multline*}
Then by similar arguments as Theorem \ref{thm_AP}, we obtain
\begin{multline*}\Big(\int_\Omega\mathcal{M}(f_1\sigma_1,f_2\sigma_2)^p v d\mu\Big)^{\frac{1}{p}}\\
\leq32\cdot2^{\frac{1}{p}}p'_1p'_2[v,\omega_1, \omega_2]^{\frac{1}{p}}_{A_{\overrightarrow{p}}}[\omega_1, \omega_2]^{\frac{1}{p}}_{W^{\infty}_{\overrightarrow{p}}}\|f_1\|_{L^{p_1}(\sigma_1)}\|f_2\|_{L^{p_2}(\sigma_2)}.
\end{multline*}
This
finishes the proof.
\end{proof}

\end{document}